\documentclass[11pt]{amsart}

\usepackage{amsmath,amsthm, amscd, amssymb, amsfonts}
\usepackage[all]{xy}
\usepackage{verbatim}
\usepackage{color}
\usepackage{array}
\usepackage{multirow}
\usepackage{hhline}

\newtheorem{theorem}{Theorem}[section]
\newtheorem{lem}[theorem]{Lemma}
\newtheorem{cor}[theorem]{Corollary}

\newtheorem{pro}[theorem]{Proposition}

\newtheorem*{maintheorema}{Main Theorem 1}
\newtheorem*{maintheoremb}{Main Theorem 2}

\theoremstyle{definition}
\newtheorem{fed}[theorem]{Definition}
\newtheorem{exa}[theorem]{Example}

\theoremstyle{remark}
\newtheorem{rem}[theorem]{Remark}

\newcommand{\ydh}{{}^{H}_{H}\mathcal{YD}}
\newcommand{\ydhdual}{{}^{H^*}_{H^*}\mathcal{YD}}

\def\pf{\begin{proof}}
\def\epf{\end{proof}}

\newcommand{\ydgdual}{{}_{\ku^G}^{\ku^G}\mathcal{YD}}
\newcommand{\nc}{\newcommand}
\nc{\D}{\Delta}
\nc{\cP}{\mathcal{P}} \nc{\cU}{\mathcal{U}} \nc{\cX}{\mathcal{X}}
\nc{\cE}{\mathcal{E}} \nc{\cS}{\mathcal{S}} \nc{\cA}{\mathcal{A}}
\nc{\cC}{\mathcal{C}} \nc{\cO}{\mathcal{O}} \nc{\cQ}{\mathcal{Q}}
\nc{\cB}{\mathcal{B}} \nc{\cJ}{\mathcal{J}} \nc{\cI}{\mathcal{I}}
\nc{\cM}{\mathcal{M}} \nc{\cG}{\mathcal{G}} \nc{\cL}{\mathcal{L}}
\nc{\cK}{\mathcal{K}}

\nc{\e}{\varepsilon}

\nc{\ba}{\mathbf{a}} \nc{\bb}{\mathbf{b}} \nc{\bt}{\mathbf{t}}
\nc{\bi}{\mathbf{i}} \nc{\bj}{\mathbf{j}}

\nc{\bB}{\mathbf{B}} \nc{\bS}{\mathbf{S}}  

\newcommand{\Sn}{{\mathbb S}}
\newcommand\gA{\mathfrak{A}}

\newcommand\id{\operatorname{id}}

\newcommand{\ydg}{{}^{\ku G}_{\ku G}\mathcal{YD}}

\newcommand\supp{\operatorname{supp}}

\newcommand\Inn{\operatorname{Inn}}

\newcommand\sgn{\operatorname{sgn}}

\newcommand\Aut{\operatorname{Aut}}

\newcommand\gr{\operatorname{gr}}

\newcommand\co{\operatorname{co}}

\newcommand\ord{\operatorname{ord}}

\newcommand\Alg{\operatorname{Alg}}

\newcommand\Aff{\operatorname{Aff}}

\def\bD{\mathbb{D}}

\def\k{\Bbbk}
\def\ku{\Bbbk}

\def\ot{\otimes}

\def\s{\mathbb{S}}

\def\N{\mathbb{N}}
\def\Z{\mathbb{Z}}
\def\F{\mathbb{F}}
\def\A{\mathbb{A}}
\def\Bb{\mathbb{B}}

\def\B{\mathfrak{B}}
\def\fop{\mathfrak{op}}
\def\fc{\mathfrak{c}}

\def\eps{\varepsilon}

\def\mA{\mathcal{A}}
\def\mG{\mathcal{G}}
\def\mJ{\mathcal{J}}
\def\mT{\mathcal{T}}

\def\Ss{\mathcal{S}}
\def\mL{\mathcal{L}}

\def\mR{\mathcal{R}}

\def\mH{\mathcal{H}}

\newcommand\can{\operatorname{can}}
\newcommand\op{\operatorname{op}}
\newcommand\cop{\operatorname{cop}}

\renewcommand{\_}[1]{_{\left( #1 \right)}}

\begin{document}



\title[Pointed or Copointed Hopf algebras]{Finite-dimensional Pointed or Copointed Hopf
algebras over affine racks}

\author[Garc\'ia Iglesias, Vay]{Agust\'in Garc\'ia Iglesias and Cristian Vay}

\address{FaMAF-CIEM (CONICET), Universidad Nacional de C\'ordoba,
Medina A\-llen\-de s/n, Ciudad Universitaria, 5000 C\' ordoba, Rep\'ublica Argentina.}

\email{(aigarcia|vay)@famaf.unc.edu.ar}

\thanks{This work was partially supported by ANPCyT-FONCyT, CONICET,
MinCyT (C\'ordoba) and Secyt (UNC). Part of the work of A.G.I. was done as a
fellow of the Erasmus Mundus programme of the EU in the Universit\`a degli Studi di
Padova.
}

\begin{abstract}
We study the pointed or copointed liftings of Nichols algebras associated to
affine racks and constant cocycles for any finite group admitting a principal 
YD-realization of these racks. In the copointed case we complete
the classification for the six affine racks whose Nichols algebra is known to be of
finite dimension. In the pointed case our method allows us to finish
four of them. In all of the cases the Hopf algebras obtained turn out to be cocycle
deformations of their associated graded Hopf algebras. All of them are new examples of
finite-dimensional copointed or pointed Hopf algebras over non-abelian groups. 
\end{abstract}

\maketitle

\section{Introduction}

Let $\ku$ be an algebraically closed field of characteristic zero and let $H$ be a
semisimple Hopf algebra over $\ku$.
This work is in the framework of the classification of finite-dimensional Hopf algebras 
whose coradical is a Hopf subalgebra isomorphic to $H$. Let $\mathfrak{F}_H$ be the family
of such Hopf algebras. This problem has two interrelated sub-problems:
\begin{itemize}
\item To determine  all $V\in \ydh$ such that the Nichols algebra $\B(V)$ is
finite-dimensional and give a presentation of $\B(V)$.
\item To classify the lifting Hopf algebras of $\B(V)$ over $H$.
\end{itemize}
If $A\in\mathfrak{F}_H$ is generated in degree one, then $A$ is a lifting of a Nichols
algebra over its coradical. It was conjectured that this holds when $H$ is a group
algebra \cite{AS1}. These steps compose the Lifting Method of \cite{AS3}.

\medbreak

First defined by Nichols, and also called quantum symmetric algebras, Nichols algebras
are determined by a profound combinatorial behavior which is no yet fully
understood. They are not Hopf algebras in the usual sense, but rather Hopf
algebras in the category of Yetter-Drinfeld modules $\ydh$.

Let $G$ be a finite group. If $G$ is abelian, all $V\in \ydg$ with
$\dim\B(V)<\infty$ have
been determined in \cite{H} and the presentation of $\B(V)$ together with a positive
answer to the conjecture in \cite{AS1} were given in \cite{A1,A2}.  If $G$
is non-abelian, it has been shown that for many (simple) groups most $V\in \ydg$ yield
Nichols
algebras of infinite dimension \cite{AFGV,AFGV2}. Furthermore, only a few examples of
finite-dimensional Nichols algebras are known, see below. Up to date, it is very
complicated to find the relations defining the Nichols algebras and to compute their
dimension, even using the computer, see \cite{grania}. Notice that Nichols
algebras in $\ydgdual$ and $\ydg$ coincide since
these categories are braided equivalent, see Section \ref{subsec:relacion entre nichols}. 

\medbreak

Recall that the Hopf algebras in $\mathfrak{F}_{\ku G}$ are called {\it pointed}, while
those in $\mathfrak{F}_{\ku^G}$ are called {\it copointed}, cf. \cite{AV}. 

\medbreak

The most prominent result in the classification of Hopf algebras is in \cite{AS3} where
the pointed Hopf algebras over an abelian group of order coprime with $210$ are
classified. The classification of nontrivial, {\it i.e.} different from group
algebras,
pointed Hopf algebras over non-abelian group is known for: $\Sn_3$ \cite{AHS}, $\Sn_4$
\cite{GG} and $\bD_{4t}$ \cite{FG}. Also they have been classified the cases
$\A_n$, $n\geq5$, and most simple sporadic groups but all turn out to be group algebras
\cite{AFGV,AFGV2,FV}. In the copointed case the classification is known only for $\Sn_3$
\cite{AV}. The Hopf algebras obtained in the above results are all liftings of Nichols
algebras over their coradical and cocycle deformations of each other
\cite{Ma3,FG,AV2,GIM,GaM}.

Also, in \cite{CDMM} the liftings of the quantum line over four
families of nontrivial semisimple Hopf algebras are classified, and in
\cite{ardizzonimstefan} another approach for the lifting problem is proposed.

\medbreak

Nichols algebras of finite-dimension over non-abelian groups appear associated to racks
and 2-cocycles, see \cite{AG2}. It is worth mentioning that racks appear also
in the calculus of knot invariants \cite{GT}. Next, we list all pairs of
non-abelian indecomposable racks and cocycles whose
associated Nichols algebras are known to be finite-dimensional, see for instance 
\cite{grania}.

(1) Racks of the conjugacy classes $\cO_m^n$ of $m$-cycles in $\Sn_n$:
\begin{itemize}
\item The rack $\cO_2^n$ and constant 2-cocycle $-1$, $n=3, 4, 5$.
\item The rack $\cO_2^n$ and a non-constant 2-cocycle $\chi$, $n=4, 5$.
\item The rack $\cO_4^4$ and constant 2-cocycle $-1$.
\end{itemize}
Their Nichols algebras were studied in \cite{MS,FK,AG2,GG}. In
\cite{V} it is shown that the Nichols algebras associated to $\cO_2^n$ with constant and
non-constant 2-cocycle are twist equivalent. All of these racks can
be realized over the symmetric groups and their duals. The families $\mathfrak{F}_{\s_3}$
and $\mathfrak{F}_{\s_4}$ were classified in \cite{AHS, GG} respectively, and
$\mathfrak{F}_{\ku^{\s_3}}$ in \cite{AV}.  

(2) The affine racks: 
\begin{itemize}
\item $(\F_3,2)$, $(\F_4,\omega)$, $(\F_5,2)$, $(\F_5,3)$, $(\F_7,3)$
and $(\F_7,5)$ with constant 2-cocycle $-1$ \cite{MS,grania1,AG2}\footnote{As racks 
$(\F_3,2)\simeq\cO_2^3$, $(\F_5,2)^*\simeq(\F_5,3)$ and $(\F_7,3)^*\simeq(\F_7,5)$}. 
\item $(\F_4,\omega)$ and a non-constant 2-cocycle $\zeta$ \cite{HLV} with $\zeta_{ii}$ a
third root of 1 for $i\in\F_4$.
\end{itemize}
The aim of this work is to study both the pointed and copointed lifting of the Nichols
algebras associated to these affine racks $(\F_b,N)$ with constant 2-cocycle -1. In this 
case no liftings are known, apart from the
case $(\F_3,2)$, see \cite[Theorem 3.8]{AG3}. In \cite{AAnMGV} a general strategy to
classify the family
$\mathfrak{F}_H$ is developed showing at the same time that they are cocycle 
deformations of the bosonization $\B(V)\# H$. We adapt the ideas there to compute the 
pointed, and
copointed, liftings of these Nichols algebras over any group $G$. We also give results
which apply to other racks. 

\medbreak

The classification in the pointed case is given by the next theorem.

\begin{maintheorema}
Let $G$ be a finite group. The pointed Hopf algebras over $\ku G$ whose infinitesimal braiding 
arises from a principal YD-realization of an affine rack $X$ with the constant 2-cocycle 
$q\equiv-1$ are classified in
\begin{itemize}
\item[(i)] Theorem \ref{teo:f3}, if $X=(\F_3,2)$. 
\item[(i)] Theorem \ref{teo:f4}, if $X=(\F_4,\omega)$.
\item[(iii)] Theorem \ref{teo:f5}, if $X=(\F_5,2)$.
\item[(iv)] Theorem \ref{teo:f5-3}, if $X=(\F_5,3)$.
\end{itemize}
All of these liftings are cocycle deformations of $\B(X,-1)\#\ku G$.
\end{maintheorema}

The first item is already in \cite[Theorem 3.8]{AG3}, without the statement about cocycle
deformations. It is important to remark that in some of these new examples, some of the
relations are not
only deformed by elements in the coradical, but also by elements in higher terms of the
coradical filtration. This phenomenon is quite new, and was only present previously in
some deformations in \cite{He} for the abelian case. 

\medbreak

The classification in the copointed case is given by the next theorem.

\begin{maintheoremb}
Let $G$ be a finite group. The copointed Hopf algebras over $\ku^G$ whose infinitesimal
braiding arises from a principal YD-realization of an affine rack $X$ with the
constant 2-cocycle $q\equiv-1$ are classified in
\begin{itemize}
\item[(i)] Theorem \ref{prop: clas of F3 2}, if $X=(\F_3,2)$.
\item[(ii)] Theorem \ref{thm:lifting over affines grandes}, if $X=(\F_4,\omega)$, 
$(\F_5,2)$, 
$(\F_5,3)$, $(\F_7,3)$ or
$(\F_7,5)$.
\end{itemize}
All of these liftings are cocycle deformations of $\B(X,-1)\#\ku^G$.
\end{maintheoremb}

We explicitly define biGalois objects to prove the last assertion. These liftings are new
examples of Hopf algebras.

\medbreak

The Hopf algebras found are presented as quotients of bosonizations of tensor
algebras. Hence the greatest obstacle to achieve our principal results is to show that
these quotients have the right dimension, or just to show that they are nonzero. The same
issue is present in the rest of the works cited above. We are able to avoid this obstacle
by showing that the quotient is a cocycle deformation, as proposed in \cite{AAnMGV}.
However, some very complicated computations are necessary at an intermediate step and we
are forced to appeal to computer program \cite{GAP}. However, we find that the computer 
is not always enough and some examples
cannot be attacked with this method. The same computational impediment is present in the 
calculation
of Nichols algebras themselves. Hence, new tools are required to attack these problems,
such as representation theory, see for instance \cite{AV2,GG,FG}.

\medbreak

The paper is organized as follows: In Section \ref{sec:pre} we give some conventions and 
notations. In Section \ref{subsec:relacion entre nichols} we give the correspondence 
between Nichols algebras in braided equivalent categories of
Yetter-Drinfeld modules. We recall the notions of rack and Yetter-Drinfeld realization of 
a rack over a group. In Section \ref{subsec:affine-racks}, we introduce the known 
examples 
of finite-dimensional Nichols algebras attached to an affine rack and give some 
properties of these which will be useful for us. In Section \ref{very general results} we 
go through the ideas in \cite{AAnMGV} and adapt them to prove new results
that apply in our setting. In Sections 6 and 7 we use these results to prove our main 
theorems. We also include an Appendix with the ideas behind some of the
computations.

\subsection*{Acknowledgments.} We thank Nicol\'as Andruskiewitsch for suggesting us this
problem and for the many useful comments he shared with us in previous versions of this
work. We also thank Iv\'an Angiono and Leandro Vendramin for very useful discussions.
In particular, the idea for computing coproducts with \cite{GAP} came up in conversations
with L. V. Also, A. G. I. specially thanks Giovanna Carnovale from the
Universit\`a degli Studi di Padova where part of his work was done and for her warm
hospitality and help. We are also indebted to Mart\'in Mombelli for sharing his 
knowledge about braided categories with us.

\section{Preliminaries}\label{sec:pre}

We work over an algebraically closed field $\k$ of characteristic zero;
$\k^*:=\k\setminus\{0\}$.
If $X$ is a set, then $\ku X$ denotes the
free vector space over $X$. If $A$ is an algebra and $X\subset A$,
then $\langle X\rangle$ is the two-sided ideal generated by $X$.

Let $G$ be a finite group. We denote by $e$ the identity element of $G$, by $\ku G$ its
group algebra and by $\ku^G$ the function algebra on $G$. The usual basis of $\ku G$ is
$\{g:g\in
G\}$ and $\{\delta_g:g\in G\}$ is its dual basis in $\ku^G$, i.e.
$\delta_g(h)=\delta_{g,h}$ for all $g,h\in G$. If $M$ is a $\ku^G$-module and $g\in G$,
the isotypic component of weight $g$ is $M[g]=\delta_g\cdot M$. We write $\supp M=\{g\in
G:M[g]\neq0\}$ and $M^\times=\bigoplus_{g\neq e}M[g]$. The symmetric group in $n$ letters 
is
denoted by $\Sn_n$ and $\sgn:\Sn_n\to\Z_2$ denotes the morphism given by the sign. 

Let $H$ be a Hopf algebra. Then $\Delta$, $\e$, $\mathcal{S}$ denote respectively the
comultiplication, the counit and the antipode. We use Sweedler's notation for
comultiplication and coaction but dropping the summation symbol. We denote by
$\{H_{[i]}\}_{i\geq0}$ the coradical filtration of $H$ and by $\gr H=\oplus_{n\geq0}\gr^n
H=\bigoplus_{n\geq0}H_{[n]}/H_{[n-1]}$ the associated graded Hopf algebra of $H$ with 
$H_{[-1]}=0$. 

Assume $\mathcal{S}$ is bijective and let $\ydh$ be the category of Yetter-Drinfeld 
modules over $H$. If $V\in\ydh$, then the dual object $V^*\in\ydh$ is defined by $\langle 
h\cdot 
f,v\rangle=\langle f,\mathcal{S}(h)\cdot v\rangle$ and $f\_{-1}\langle 
f\_{0},v\rangle=\mathcal{S}^{-1}(v\_{-1})\langle f,v\_{0}\rangle$ for all $v\in V$, $f\in 
V^*$  and $h\in H$, where $\langle\,,\,\rangle$ denotes the standard evaluation.

\subsection{Galois objects}
Let $H$ be a Hopf algebra with bijective antipode and $A$ be a right $H$-comodule algebra 
with right $H$-coinvariants $A^{\co H}=\ku$. 

If there exist a convolution-invertible $H$-colinear map $\gamma:H\to A$, then $A$ is 
called a right \emph{cleft object}. The map 
$\gamma$ can be chosen so that $\gamma(1)=1$, in which case it is called a {\it section}.
In turn, $A$ is called a right 
\emph{$H$-Galois object} 
if the following linear map is bijective:
$$
\can:A\ot A\longmapsto A\ot H,\quad a\ot b\mapsto ab\_0 \ot b\_{1}.
$$
Analogously, left $H$-Galois objects are defined. Let
$L$ be another Hopf algebra. An
$(L,H)$-bicomodule algebra is an $(L,H)$-biGalois object if it is simultaneously a left 
$L$-Galois object and a right $H$-Galois object.

Assume $A$ is a right $H$-Galois object. There is an associated Hopf algebra $L(H,A)$ such 
that $A$ is a
$(L(A,H),H)$-biGalois object, see \cite[Section 3]{S}. $L(A,H)$ is a subalgebra
of
$A\ot
A^{\op}$. Moreover, if $L$ is a Hopf algebra
such that $A$ is $(L,H)$-biGalois then $L\cong  L(A,H)$. More precisely, if
$\delta$, $\delta_L$ stand for the coactions of $L(A,H)$ and $L$ in $A$, there is
a
Hopf algebra isomorphism $F:L(A,H)\to L$ such that $\delta_L=(F\ot \id)\delta$ and 
\begin{equation}\label{eqn:F}
F\left(\sum a_i\ot b_i\right)\ot 1_{A}=\sum\lambda_L(a_i)(1\ot b_i), \quad \sum a_i\ot
b_i\in L(A,H).
\end{equation}
Thus, one can use Galois objects  to find new examples  of Hopf algebras. Furthermore, 
$L(H,A)$ is a cocycle deformation of $H$ \cite[Theorem 
3.9]{S}.

\section{Nichols algebras and Racks}\label{subsec:relacion entre nichols}

From now on $\cC$ denotes a category of (left, right or left-right) Yetter-Drinfeld 
modules
over a finite-dimensional Hopf algebra $H$. Then $\cC$ is a 
braided monoidal category. Let
$c$ be the canonical braiding of $\cC$. See {\it e.~ g.} \cite{K} for details about
braided monoidal categories.

\smallbreak

Let $V\in\cC$. The tensor algebra $T(V)$ is an algebra in $\cC$. Also, $T(V)\ot T(V)$ is 
an algebra with multiplication
$(m\ot m)\circ(\id\ot\, c\ot\id)$.
Hence $T(V)$ becomes a Hopf algebra in $\cC$ extending by the universal property
the following maps
$$
\Delta(v)=v\ot1+1\ot v,\quad\varepsilon(v)=0\quad\mbox{and}\quad\cS(v)=-v,\quad v\in
V.
$$                                                            
Let $\cJ(V)$ be the largest 
Hopf ideal of $T(V)$ generated as an ideal by homogeneous 
elements of degree $\geq 2$.

\begin{fed}\cite[Proposition 2.2]{AS2}
The {\it Nichols algebra} of $V$ (in $\cC$) is $\B(V)=T(V)/\cJ(V)$. 
\end{fed}

See \cite{AS2} for details about Nichols algebras. Let $n\in\N$; we denote by 
$\mJ^n(V)$, resp. $\B^n(V)$, the homogeneous component of 
degree $n$ 
of $\cJ(V)$, resp. of $\B(V)$. We set $\mJ_n(V)=\langle\bigoplus_{l=2}^n 
\mJ^l(V)\rangle$ and $\widehat{\B_n}(V)=T(V)/\mJ_n(V)$.

Let $A$ be Hopf algebra such that $\gr A$ is isomorphic to $\B(V)\# H$. Then $A$ is called 
a {\it lifting of $\B(V)$ over $H$}. The {\it infinitesimal braiding of $A$} is 
$V\in\ydh$ with the braiding of $\ydh$. Recall from \cite[Proposition 2.4]{AV} 
that there 
exists a {\it lifting map} $\phi:T(V)\#H\rightarrow A$, that is an epimorphism of Hopf
algebras such that
\begin{align}\label{eq:properties of A and phi}
\phi_{|H}=\id,&& \phi_{|V\#H}\mbox{ is injective}&&\mbox{ and }&&\phi((\ku\oplus
V)\#H)=A_{[1]}.
\end{align}

\bigbreak

We recall another characterization of $\cJ(V)$, see {\it e.~ g.} \cite{AG1, AS2}.
Fix $n\in\N$. Let $\Bb_n$ be the {\it Braid group}: It is generated by 
$\{\sigma_i:1\leq i< n\}$
subject to the relations
$
\sigma_i\sigma_{i+1}\sigma_i=\sigma_{i+1}\sigma_i\sigma_{i+1}$
and $\sigma_i\sigma_j=\sigma_{j}\sigma_i$ for all $1\leq i, j< n$ such that
$|i-j|>1$. The projection $ \Bb_n\twoheadrightarrow\Sn_n$, $\sigma_i\mapsto(i\,i+1)$,
$1\leq i< n$, admits a set-theoretical section $s:\Sn_n\rightarrow\Bb_n$ defined by
$s(i\,i+1)=\sigma_i$, $1\leq i< n$, and $s(\tau)=\sigma_{i_1}\cdots\sigma_{i_{\ell}}$, if
$\tau=(i_1\,i_1+1)\cdots(i_{\ell}\,i_{\ell}+1)$ with $\ell$ minimum; this is the
{\it Matsumoto section}. The {\it quantum symmetrizer} is:
$$
\bS_n=\sum_{\tau\in\Sn_n}s(\tau)\in\ku\Bb_n.
$$
The group $\Bb_n$ acts on $V^{\ot n}$ via the assignment $\sigma_i\mapsto c_{i,i+1}$, 
$1\leq i< n$, where $c_{i,i+1}:V^{\ot n}\longrightarrow V^{\ot n}$ is the morphism
$$
\id\ot\, c\ot\id:V^{\ot i-1}\ot V^{\ot2}\ot V^{\ot n-i-1}\longrightarrow V^{\ot i-1}\ot
V^{\ot2}\ot V^{\ot n-i-1}.
$$
Then the homogeneous components of $\cJ(V)$ are given by
$$
\cJ^k(V)=\ker\bS_k,\quad k\in\N.
$$

\subsection{Correspondence between Nichols algebras in braided equivalent categories}
Let $H$, $\cC$ be as above. Let $H'$ be a finite-dimensional Hopf algebra, $\cC'$ be a
category of Yetter-Drinfeld modules over $H'$. Assume there is a
functor $(F,\eta):\cC\rightarrow\cC'$ of braided monoidal categories, {\it i. ~e.}
$F:\cC\rightarrow\cC'$ is a functor and $\eta:\ot\circ F^2\rightarrow F\circ\ot$ is a
natural isomorphism such that the diagrams 
\begin{align}\label{eq:F y eta}
\xymatrix{
F(U)\ot F(V)\ot F(W)\ar@{->}[rr]^{\eta\,\ot\, \id}\ar@{->}[d]_{\id\ot\,\eta} && F(U\ot
V)\ot F(W)\ar@{->}[d]^{\eta}&\\
F(U)\ot F(V\ot W)\ar@{->}[rr]_{\eta} && F(U\ot V\ot W),&
}
\end{align}
\begin{align}\label{eq:c y eta}
\xymatrix{
F(U)\ot F(V)\ar@{->}[rr]^{c_{F(U),F(V)}}\ar@{->}[d]_{\eta} && F(V)\ot 
F(U)\ar@{->}[d]^{\eta}&\\
F(U\ot V)\ar@{->}[rr]_{F(c_{U,V})} && F(V\ot U),
}
\end{align}
commute for each $U,V,W\in\cC$. 

Fix $V\in\cC$. For $m,n\in\N$, set $\eta_{m,n}=\eta_{V^{\ot m},
V^{\ot n}}$ and
\begin{align*}
\eta_n&=\eta_{n-1,1}(\eta_{n-2,1}\ot\id)\cdots(\eta_{2,1}\ot\id)(\eta\ot\id):F(V)^{\ot
n}\longrightarrow F(V^{\ot n}).
\end{align*}
By abuse of notation, we still write $\eta=\eta_{1,1}=\eta_2$. By \eqref{eq:F y eta}, it 
holds that
\begin{align}\label{eq:asociatividad de los etas}
\eta_{m+n+k}=\eta_{m,n+k}\,(\id\ot\eta_{n,k})\,(\eta_m\ot\eta_n\ot\eta_{
k}), \quad m,n,k\in\N. 
\end{align}
Note that $\Bb_n$ acts on $F(V^{\ot n})$ via $\sigma_i\mapsto
F(c_{i,i+1})$. Then the commutative diagram \eqref{eq:c y eta} implies that
$\eta$ is an isomorphism of $\Bb_2$-modules. Moreover, combining \eqref{eq:F y eta} and
\eqref{eq:c y eta} with the fact that $\eta$ is a natural isomorphism, we obtain that
$\eta_n:F(V)^{\ot n}\longrightarrow F(V^{\ot n})$ is an isomorphism of $\Bb_n$-modules in
$\cC'$. As a consequence we have the next lemma.

\begin{lem}\label{thm:generators of nichols algebras via functors general}
Assume $(F,\eta):\cC\rightarrow\cC'$ is exact. Let $V\in\cC$ with $\dim V<\infty$. The
ideals defining the Nichols algebras $\B(V)$ and $\B(F(V))$ are related by
$$
\cJ^n(F(V))=\eta_n^{-1}F(\cJ^n(V))\,\mbox{ for all }n\in\N.
$$
If $F$ preserves dimensions, then $\dim \B^n(V)=\dim \B^n(F(V))$ for all $n\in\N$.
\end{lem}

\begin{proof}
Recall that $\cJ^n(F(V))$ is the kernel of $\bS_n$ acting on $F(V)^{\ot n}$, $n\in\N$.
Since $F$ is exact and $\eta_n$ is an isomorphism, the theorem follows.
\end{proof}

We can apply the above result to the categories $\ydh$ and $\ydhdual$. In fact, by
\cite[Proposition 2.2.1]{AG1} they are braided equivalent monoidal categories via the
functor $(F,\eta)$ defined as follows: $F(V)=V$ as a vector space,
\begin{align}\label{prop:equiv de cat gdual}
\begin{split}
f\cdot v&=\langle f,\cS(v\_{-1})\rangle 
v\_{0},\quad\delta(v)=f_i\ot\cS^{-1}(h_i)v\quad\mbox{and}\\
\noalign{\smallbreak}
\eta:\, &F(V)\ot F(W)\longmapsto F(V\ot W),\quad v\ot w\mapsto w_{(-1)}v\ot w_{(0)}
\end{split}
\end{align}
for every $V,W \in \ydh$, $f \in H^*$, $v \in V$,
$w\in W$. Here $\{h_i\}$ and $\{f_i\}$ are dual bases of $H$ and $H^*$.

\begin{lem}\label{lem:generators of nichols algebras via functors} 
Let $V\in\ydh$ of finite dimension and $M\subset V^{\ot n}$ in $\ydh$. Let 
$N=\bigoplus_{m\in\N} N^m$ with $N^m\subset V^{\ot m}$ in $\ydh$, $m\in\N$. Then
\begin{enumerate}\renewcommand{\theenumi}{\alph{enumi}}\renewcommand{\labelenumi}{
(\theenumi)} 
\item\label{eq:lem:generators of nichols algebras via functors tensor} $F(V)^{\ot 
m}\ot\eta_n^{-1}F(M)\ot F(V)^{\ot k}=(\eta_{m+n+k})^{-1}F(V^{\ot m}\ot
M\ot V^{\ot k})$.
\smallbreak
\item $\langle\eta_n^{-1}F(M)\rangle=\sum_{m,k}(\eta_{m+n+k})^{-1}F(V^{\ot m}\ot M\ot
V^{\ot k}).$
\smallbreak
\item\label{eq:lem:generators of nichols algebras via functors cocientes} Let 
$\overline{M}\subset
T(V)/\langle N\rangle$. In
$T(F(V))/\langle\bigoplus_m\eta_m^{-1}F(N^m)\rangle$ it holds that 
$\overline{\eta_n^{-1}F(M)}=\eta_n^{-1}F(\overline{M})$.
\end{enumerate}
\end{lem}

\begin{proof}
(a) Let $x\in V^{\ot m}$, $r\in M$ and $y\in V^{\ot k}$. By \eqref{eq:asociatividad de los
etas}, there exist $x'\in V^{\ot m}$, $r'\in M$ and $y'\in V^{\ot k}$ such that
$(\eta_{m+n+k})^{-1}(x\ot r\ot y)=\eta_m^{-1}(x')\ot \eta_n^{-1}(r')\ot \eta_{k}^{-1}
(y')$. Also by \eqref{eq:asociatividad de los etas}, there exist $x''\in V^{\ot m}$,
$r''\in M$ and $y''\in V^{\ot k}$ such that $\eta_{m+n+k}(x\ot r\ot y)=x''\ot r''\ot y''$.
Since $(\eta_{m+n+k})^{\pm1}_{|V^{\ot m}\ot M\ot V^{\ot k}}$ are injective
morphisms the statement follows. (b) and (c) are straightforward.
\end{proof}

Lemma \ref{lem:generators of nichols algebras via 
functors} \eqref{eq:lem:generators of nichols algebras via functors cocientes} is useful to 
find deformations of Nichols algebras. Next lemma is a consequence of Lemma \ref{lem:generators of nichols algebras via 
functors} \eqref{eq:lem:generators of nichols algebras via functors tensor}.

\begin{lem}\label{thm:generators of nichols algebras via functors} 
Let $M=\bigoplus_{m\in\N} M^m$ with $M^m\subset V^{\ot m}$ in $\ydh$, $m\in\N$. Assume 
that $M$ generates $\cJ(V)$ as an ideal. Then
\begin{enumerate}\renewcommand{\theenumi}{\alph{enumi}}\renewcommand{\labelenumi}{
(\theenumi)} 
\item\label{eq:generadors thm:generators of nichols algebras via functors}
$\bigoplus_{m\in\N}\eta_m^{-1}F(M^m)\in \ydhdual$ generates $\cJ(F(V))$ as an ideal.
\smallbreak
\item\label{eq:pre generatos thm:generators of nichols algebras via functors} 
$\cJ_k(F(V))=\langle\bigoplus_{l=2}^k\eta_l^{-1}F(\cJ^l(V))\rangle$ for all $k\in\N$. \qed
\end{enumerate}
\end{lem}

\subsection{Racks}\label{subsec:YD realization}
A {\it rack} is a nonempty set $X$ with an operation $\rhd:X\times X\rightarrow
X$ such that 
$$\phi_i:X\longmapsto X,\, j\mapsto i\rhd j,$$
is a bijective map and $\phi_i(j\rhd k)=\phi_i(j)\rhd\phi_i(k)$ for all $i,j,k\in X$. The 
subgroup of $\mathbb{S}_{X}$ generated  by $\{\phi_i\}_{i\in X}$ is denoted $\Inn _\rhd 
X$, it is a
subgroup of the group of rack automorphisms $\Aut_\rhd X$. 

A function $q:X\times X\rightarrow\ku^*$, $(i,j)\mapsto q_{ij}$, is a ({\it rack}) {\it 
2-cocycle} if $q_{i,j\rhd k}q_{j,k}=q_{i\rhd j,i\rhd k}q_{i, k}$ for all $i,j,k\in X$. We 
refer to \cite{AG2} for details about racks.

\begin{fed}\cite[Definition 3.2]{AG3}, \cite[Subsection 5]{MS}. Let $X$ be a rack and $q$ 
be a $2$-cocycle on $X$. A \emph{principal 
YD-realization} of $(X,q)$ over a finite group $G$ is a collection $(\cdot, g, 
\{\chi_i\}_{i\in X})$ where
\begin{itemize}
\item $\cdot$ is an action of $G$ on $X$;
\smallbreak
\item $g:X\longmapsto G$, $i\mapsto g_i$, is a function such that $g_{h\cdot i} = hg_ 
{i}h^{-1}$ and $g_{i}\cdot j=i\rhd j$ for all
$i,j\in X$ and $h\in G$;
\smallbreak
\item\label{item:1-cocycle} $\{\chi_i\}_{i\in X}$ is a {\it 1-cocycle} -- that is a family
of maps
$\chi_i:G\to\k^*$ such that $\chi_i(ht)=\chi_{t\cdot i}(h)\chi_i(t)$ for all $i\in
X$, $h,t\in G$ --  satisfying $\chi_i(g_{j})=q_{ji}$ for all
$i,j\in X$.
\end{itemize}
We will assume that all realizations are \emph{faithful}, that is $g$ is injective. 
\end{fed}

These data define an object $V(X,q)\in\ydg$ \cite{AG3}. Namely, as a vector space 
$V(X,q)=\ku \{x_i\}_{i\in X}$, the action and coaction are
\begin{align}\label{eqn:yetter-drinfeld}
t\cdot x_{i}=\chi_i(t)x_{t\cdot i}\quad\mbox{and}\quad\lambda(x_i)=g_i\ot
x_i,\quad t\in G, i\in X.
\end{align}
We denote by $T(X,q)$ the tensor algebra of $V(X,q)$, its Nichols algebra is 
denoted by $\B(X,q)$ and the defining ideal is $\cJ(X,q)$.

\smallbreak

Let $W(q,X)$ be the object in $\ydgdual$ obtained by applying the functor 
\eqref{prop:equiv de cat gdual} to the above Yetter-Drinfeld module $V(X,q)$ over $\ku G$. Then
\begin{align}\label{eqn:yetter-drinfeld-dual}
\delta_t\cdot x_i=\delta_{t,g_i^{-1}}x_i\quad\mbox{and}\quad \lambda(x_i)=\sum_{t\in
G}\chi_i(t^{-1})\delta_t\ot x_{t^{-1}\cdot  i}, \quad t\in G, i\in X.
\end{align}
We denote by $T(q,X)$ the tensor algebra of $W(q,X)$, its Nichols algebra is 
denoted by $\B(q,X)$ and the defining ideal is $\cJ(q,X)$.

\smallbreak

Note that the smash product Hopf algebra $T(X,q)\#\ku G$ satisfies
\begin{align}\label{eq:TV smash ku de G}
tx_i=\chi_i(t)x_{t\cdot i}t \quad\mbox{and}\quad\Delta(x_i)=x_i\ot 1 +{g_i} \ot
x_i,\quad t\in G, i\in X.
\end{align}
The smash product Hopf algebra $T(q,X)\#\ku^G$ satisfies for all $t\in G$, $i\in X$
\begin{align}\label{eq:TV smash ku a la G}
\delta_tx_i=x_i\delta_{g_it}\quad\mbox{and}\quad\Delta(x_i)=x_i\ot 1 +\sum_{t\in
G}\chi_i(t^{-1})\delta_t\ot x_{t^{-1}\cdot
i}.
\end{align}

To find all the groups $G$ supporting a principal YD-realization of $(X,q)$ presents
hard computational aspects \cite[Section 3]{AG3}, see {\it e.~g.} Lemma \ref{le:the
enveloping se proyecta} (c) below. A possible approach is the
following. Let $F(X)$ be the free group generated by $\{g_i\}_{i\in X}$. The {\it
enveloping group $G_X$ of $X$}, see \cite{EG,J}, is
\begin{align}\label{eq:enveloping group}
G_X=F(X)/\langle g_ig_j-g_{i\rhd j}g_i:\, i,j\in X\rangle.
\end{align}

If $X$ is finite and indecomposable, then the order $n$ of $\phi_i$ does not depend on
$i\in X$ and is called the {\it degree} of the rack, see \cite[Definition 2.18]{HLV}, also
\cite{GLV}. Thus, there is a series of finite versions of $G_X$, given by 
$$
G^k_X=G_X/\langle g_i^{kn},\, i\in X\rangle, \quad k\in\N.
$$
$G^1_X$ is denoted by $\overline{G_X}$ and called the finite enveloping group of $X$ in
\cite{HLV}.

\begin{lem}\label{le:the enveloping se proyecta}
Let $X$ be a faithful and indecomposable rack of degree $r$ with a 2-cocycle $q$. Let
$(\cdot,g,\{\chi_i\}_{i\in X})$ be a principal YD-realization of $(X,q)$ over a finite
group $G$ and $K\subset G$ be the subgroup generated by $\{g_i:\,i\in X\}$. Then
\begin{enumerate}\renewcommand{\theenumi}{\alph{enumi}}\renewcommand{\labelenumi}{
(\theenumi)}
\item\label{item:le:the enveloping se proyecta} $K$ is normal and a
quotient of $G_X^r$. 
\item\label{ite:Autle:the enveloping se proyecta}\cite[Lemma 3.3(c)]{AG3} $G$ acts by rack
automorphisms on $X$.
\item\label{ite:chiG:the enveloping se proyecta}\cite[Lemma 3.3(d)]{AG3} If $q$ is
constant, then there exists a multiplicative character $\chi_G:G\rightarrow\ku^*$
such that $\chi_G=\chi_{i}$ for all $i\in X$. 
\end{enumerate}
\end{lem}

\pf
\eqref{item:le:the enveloping se proyecta} Clearly $K$ is normal. As $X$ is faithful, the 
map
$g:X\to G$ is injective and thus we have an epimorphism $F(X)\to K$. Since the relations
defining $G_X^r$ are satisfied in $K$, the epimorphism factorizes through
$G_X^r$.
\epf

\begin{lem}\label{le:neq in group rack}
Let $(X,q)$, $(\cdot,g,\{\chi_i\}_{i\in X})$ and $K$ be as in the above lemma.
\begin{enumerate}\renewcommand{\theenumi}{\alph{enumi}}\renewcommand{\labelenumi}{
(\theenumi)}
\item\label{item:general le:neq in group rack} If $i\rhd j\neq j$, then $g_i^{\ell}\neq 
g_j$ for all $\ell\in\Z$. In particular,
$g_ig_j\neq e$.
\smallbreak
\item\label{item:no F3 2 le:neq in group rack} Let $i,j\in X$ and $\ell\in\Z$ be such that 
$\phi_i^\ell(j)\neq j$. Then $g_i^\ell\neq e$.
\end{enumerate}
Assume $q\equiv \xi$ is constant, for an $n$th root of unity $\xi$.
\begin{itemize}
\item[(c)] If $n_1+\cdots+n_a\not\equiv m_1+\cdots+m_b\mod n$, then $g_{i_1}^{n_1}\cdots 
g_{i_a}^{n_a}\neq g_{j_1}^{n_1}\cdots g_{j_b}^{n_b}$.
\smallbreak
\item[(d)] $(\chi_{G|K})^n=\eps$.
\end{itemize}
\end{lem}

\begin{proof}
\eqref{item:general le:neq in group rack}-\eqref{item:no F3 2 le:neq in group rack} We 
show that if the equality holds then
$g_i=g_j$. Notice that $g_ig_j=g_{i\rhd j}g_i$ for all $i,j\in X$. 
If $g_j=g_i^\ell$, then $g_j=g_i(g_i^\ell)g_i^{-1}=g_ig_jg_i^{-1}=g_{i\rhd j}$
but $j\neq i\rhd j$. 
In particular, $g_i^{-1}\neq g_j$ and hence $e\neq g_ig_j$.
If $e=g_i^\ell$, then $g_j=g^\ell_ig_j=g_{\phi_i^\ell(j)}g_i^\ell=g_{\phi_i^\ell(j)}$ but
$j\neq\phi_i^\ell(j)$. (c) Apply the multiplicative character
$\chi_G$. (d) is immediate.
\end{proof}

\subsection{The dual rack}\label{subsec:dual rack}

Fix a finite rack $(X,\rhd)$. The {\it dual rack} $X^*$ is 
the pair 
$(X,\rhd^{-1})$ where 
$$i\rhd^{-1}j=\phi_i^{-1}(j)\quad\mbox{for all}\quad i,j\in X.$$
Fix a $2$-cocycle $q$ on $X$ and a principal YD-realization $(\cdot,g,\{\chi_i\}_{i\in 
X})$ of $(X,q)$ over a finite group $G$.
Let $q^*:X\times X\rightarrow\ku^*$ be the $2$-cocycle on $X^*$ given by
$$
q^*_{i,j}=q_{i,i\rhd^{-1}j}\quad\mbox{for all}\quad i,j\in X.
$$
Then the dual object to $V(X,q)$ in $\ydg$ (resp. $W(q,X)$ in $\ydgdual$) is isomorphic to the object 
$V(X^*,q^*)$ in $\ydg$ (resp. $W(q^*,X^*)$ in $\ydgdual$) attached to the principal YD-realization 
$(\cdot,g^{-1},\{\chi^{-1}_i\}_{i\in X})$ over $G$, see for example 
\cite[Equation (1)]{GG}.

\smallbreak

We set $q^{-*}:=(q^*)^{-1}$. It is easy to see that $q^{-*}$ is a $2$-cocycle on $X^*$ 
and that 
$(\cdot,g^{-1},\{\chi_i\}_{i\in X})$ is a principal YD-realization of $(X^*,q^{-*})$ over 
$G$. 

\smallbreak

Let $V(X,q),V(X^*,q^{-*})\in\ydg$ be defined by \eqref{eqn:yetter-drinfeld} for 
$(\cdot,g,\{\chi_i\}_{i\in X})$ and $(\cdot,g^{-1},\{\chi_i\}_{i\in X})$, respectively. We 
denote by $\{y_i\}_{i\in X}$ the basis of $V(X^*,q^{-*})$.
We define the linear map $\mathfrak{c}:T(X,q)\rightarrow T(X^*,q^{-*})$ as follows: 
$\fc(1)=1$,
\begin{align*}
\fc(x_i)&=y_i& &\mbox{if }\,i\in X\,\mbox{ and }\\
\fc(m\,r)&=\fc(m\_{0})\,(m\_{-1}\cdot\fc(r))&&\mbox{if }\,m,r\in T(X,q).
\end{align*}
It is easy to see that $\fc$ is well defined.

\begin{pro}\label{prop:rel of Xdual q over G} 
Let $S$ be a set of generators of the defining ideal $\cJ(X,q)$ of $\B(X,q)\in\ydg$. Then 
the defining ideal of $\B(X^*,q^{-*})\in\ydg$ satisfies $\cJ(X^*,q^{-*})=\fc(\cJ(X,q))$ 
and it is generated by $\fc(S)$.
\end{pro}

\begin{proof}
We consider the co-opposite Hopf algebra $(\B(X,q)\#\ku G)^{\cop}$. As $\ku G$ is 
cocommutative, $(\B(X,q)\#\ku G)^{\cop}\simeq R\#\ku G$ for some graded braided Hopf 
algebra $R\in\ydg$. Moreover, $R$ is the Nichols algebra of $\cP(R)\in\ydg$ because 
$(\B(X,q)\#\ku G)^{\cop}$ is generated as an algebra by the first term of its coradical 
filtration which is $(\B(X,q)\#\ku G)_{[1]}$.

Now, $\cP(R)=\ku\{x_i\#g_i^{-1}\}_{i\in X}$ with coaction 
$\lambda(x_i\#g_i^{-1})=g_i^{-1}\ot x_i\#g_i^{-1}$ and action 
$g\cdot(x_i\#g_i^{-1})=gx_i\#g_i^{-1}g^{-1}=\chi_i(g)x_{g\cdot i}\#g_{g\cdot i}^{-1}$ for 
all $i\in X$, $g\in G$. Then $\cP(R)\simeq V(X^*,q^{-*})$ in $\ydg$ via the assignment 
$x_i\#g_i^{-1}\mapsto y_i$ for all $i\in X$. Therefore
$$
\vartheta:(\B(X,q)\#\ku G)^{\cop}\longrightarrow\B(X^*,q^{-*})\#\ku G,\,\, x_i\#g \mapsto 
y_i\#g_ig\quad i\in X,\,g\in G
$$
is a Hopf algebra isomorphism. Let $m\in\cJ(X,q)$ be such that $m\_{-1}\ot m\_{0}=g_m\ot 
m$. Then $0=\vartheta(m)=\fc(m)\# g_m$ and hence $\fc(m)\in \cJ(X^*,q^{-*})$. This shows 
that $\fc(\cJ(X,q))\subseteq\cJ(X^*,q^{-*})$ and the other inclusion is proved in a 
similar way. The definition of $\fc$ implies the last statement.
\end{proof}

Now, we consider $W(q,X),W(q^{-*},X^*)\in\ydgdual$ according to 
\eqref{eqn:yetter-drinfeld-dual}. Let $(\,)^{\mathfrak{op}}:T(q,X)\rightarrow 
T(q^{-*},X^*)^{\op}$ be the algebra map given by $x_{i}^{\fop}=y_{i}$ for all $i\in X$, 
here $T(q^{-*},X^*)^{\op}$ is the opposite algebra of $T(q^{-*},X^*)$.

\begin{pro}\label{prop:rel of Xdual q over kalaG} 
Let $S$ be a set of generators of the defining ideal $\cJ(q,X)$ of $\B(q,X)\in\ydgdual$. 
Then the defining ideal of $\B(q^{-*},X^*)\in\ydgdual$ satisfies 
$\cJ(q^{-*},X^*)=(\cJ(q,X))^{\fop}$ and is generated by $S^{\fop}$.
\end{pro}

\begin{proof}
We consider the opposite Hopf algebra $(\B(q,X)\#\ku^G)^{\op}$. As $\ku^G$ is commutative, 
we can see that
$$
\vartheta:(\B(q,X)\#\ku^G)^{\op}\longrightarrow\B(q^{-*},X^*)\#\ku^G,\quad 
x_i\#\delta_g\mapsto y_i\#\delta_g\quad i\in X,\,g\in G.
$$
is a Hopf algebra isomorphism.
If $m\in\cJ(q,X)$, then $0=\vartheta(m)=m^{\fop}$ and hence 
$(\cJ(q,X))^{\fop}\subseteq\cJ(q^{-*},X^*)$. The other inclusion is proved in a similar 
way and the definition of $(\,)^{\fop}$ implies the last statement.
\end{proof}

\begin{pro}\label{prop:bij co and cop}
The following maps are bijective correspondences. 
\begin{align*}
\big\{\mbox{Liftings of }\,\B(X,q)\,\mbox{ over }\,\ku 
G\big\}&\longmapsto\big\{\mbox{Liftings of }\,\B(X^*,q^{-*})\,\mbox{ over }\,\ku G\big\}\\
A&\longmapsto A^{\cop},\\
\big\{\mbox{Liftings of }\,\B(q,X)\,\mbox{ over 
}\,\ku^G\big\}&\longmapsto\big\{\mbox{Liftings of }\,\B(q^{-*},X^*)\,\mbox{ over 
}\,\ku^G\big\}\\
A&\longmapsto A^{\op}.
\end{align*}
\end{pro}

\begin{proof}
We only prove the pointed case. The copointed case is similar.

Let $A$ be a lifting of $\B(X,q)$ over $\ku G$. It is enough to prove that $A^{\cop}$ is 
a lifting of $\B(X^*,q^{-*})$ over $\ku G$ since $(A^{\cop})^{\cop}=A$. Clearly $A^{\cop}$ 
is generated as an algebra by $A_{[1]}$ and $\gr(A^{\cop})=(\gr A)^{\cop}$. Then 
$A^{\cop}$ is a lifting of a Nichols algebra $\B(V)$ for some $V\in\ydg$. As in Proposition 
\ref{prop:rel of Xdual q over G}, we can see that $V\simeq V(X^*,q^{-*})\in\ydg$.
\end{proof}

\section{Nichols algebras attached to affine racks}\label{subsec:affine-racks}

Let $A$ be an abelian group and $T\in\Aut A$. The {\it affine rack} $\Aff(A,T)$ is the set $A$ 
with operation
$$
a\rhd b=T(b) +(\id-T)(a)\quad\mbox{for all}\quad a,b\in A,
$$ 
see \cite{AG2}. The dual rack $\Aff(A,T)^*$ is the affine rack $\Aff(A,T^{-1})$.

We define a family of principal YD-realizations for $\Aff(A,T)$ and
a constant 2-cocycle. Let $C_n$ be the cyclic group of order $n\in\N$ generated by $t$. If
$\ord T$
divides $n$, then $A\rtimes_T C_{n}$ is the semidirect product of $A$ and $C_n$ with
respect to $T$ where $t\cdot a=T(a)$ for $a\in A$. Let $\xi$ be a primitive root of 1 and
$\ell=[\ord T,\ord(\xi)]$ be the minimum common
multiple of $\ord T$ and $\ord\xi$.

\begin{pro}\label{pro:affine-realization}
Let $k,m\in\N$ with $0\leq k< m$. Consider the affine rack $X=\Aff(A,T)$ with constant $2$-cocycle
$\xi$. Let
\begin{itemize}
\item $g:A\longmapsto A\rtimes_T C_{m\ell}$ be the map $a\mapsto g_{a}=a\times 
t^{k\ell+1}$;
\smallbreak
\item $\cdot:(A\rtimes_T C_{m\ell})\times A\to A$ be the assignment $h\cdot a=b$, if 
$hg_ah^{-1}=g_b$;
\smallbreak
\item $\chi_{a}:A\rtimes_T C_{m\ell}\to\k^*$ be the map $\chi_{a}(b\times
t^s)=\xi^s$, for $a,b\in A$, $s\in\N$.
\end{itemize}
Then $(g,\cdot,\{\chi_a\}_{a\in A})$ is a faithful Yetter-Drinfeld realization of $(X, 
\xi)$
over $A\rtimes_T C_{m\ell}$. 
\end{pro}

A realization $(g,\cdot,\{\chi_a\}_{a\in A})$ as in Proposition 
\ref{pro:affine-realization} is called an {\it $(m,k)$-affine realization} of $(\Aff(A,T),\xi)$.

\begin{proof}
Clearly, $g$ is injective. If
$h=a\times t^s\in A\rtimes_T C_{m\ell}$ and $b\in A$, then
$hg_{b}h^{-1}=((\id-T)(a)+T^s(b))\times t^{k\ell+ 1}$.
Thus the action $\cdot$ is well defined since the image of $g$ is a conjugacy class and
$g_{a}\cdot b=a\rhd b$.
Also $\chi_{a}(g_{b})=\xi$ and $\chi_{a}=\chi_{b}$ is a group morphism for all $a,b\in A$. Then $\{\chi_{a}\}$ is a 1-cocycle.
\end{proof}

We denote by $\F_b$ the finite field of $b$ elements. The multiplication by $N\in\F_b^*$
is an automorphism which we also denote by $N$. Then $\Aff(\F_b, N)$ is faithful and 
indecomposable and satisfies
\begin{align}\label{eq:inn equal aut}
\Inn_\rhd\Aff(\F_b,N)=\F_b\rtimes_{N}C_{\ord N}=\Aut_{\rhd}\Aff(\F_b,N), 
\end{align}
the first equality is easy; the second one is by \cite[Corollary 1.25]{AG2}.

\smallbreak

Let $q$ be a 2-cocycle on $\Aff(\F_b,N)$ and let $(\cdot,g,\{\chi_i\}_{i\in X})$ be a 
principal YD-realization of 
$(\Aff(\F_b,N),q)$ over a finite group $G$. If $q$ is 
constant, pick $i\in X$ and set $\chi_G=\chi_i$, cf. Lemma \ref{le:the enveloping se 
proyecta} \eqref{ite:chiG:the enveloping se proyecta}. From now on, we denote 
$V(b,N,q)\in\ydg$ the corresponding Yetter-Drinfeld 
module as in \eqref{eqn:yetter-drinfeld}. Also, $T(b,N,q)$ and $\B(b,N,q)$ denote 
respectively its tensor algebra and the Nichols algebra with ideal of relations 
$\cJ(b,N,q)$.

\subsection{The Nichols algebras $\B(b,N,q)\in\ydg$}\label{subsec:nichols} 

We list all the known finite-dimensional Nichols algebras attached to an 
affine rack $\Aff(\F_b,N)$ and a $2$-cocycle $q$, see {\it e. ~g.} \cite{grania}.

\subsubsection{The Nichols algebra $\B(3,2,-1)$}\label{subsubsec:nicholsF3} Its ideal
$\cJ(3,2,-1)$ is generated by
\begin{align}\label{eq:F3 2}
x_i^2, \quad x_ix_j+x_{2j-i}x_i+x_jx_{2j-i},\quad i,j\in\F_3.
\end{align}
This Nichols algebra has dimension $12$ and was computed in \cite{MS,FK}.

\subsubsection{The Nichols algebra $\B(4,\omega, -1)$}\label{subsubsec:nicholsF4} Let
$\omega\in\F_4$ be such that $\omega^2+\omega+1=0$. The ideal $\cJ(4,\omega, -1)$ is
generated by $z_{(4, \omega,-1)}:=(x_\omega
x_1x_0)^2+(x_1x_0x_\omega)^2+(x_0x_\omega x_1)^2$ and
\begin{align}\label{eq:F4 omega}
x_i^2,& \quad x_ix_j+x_{(\omega+1)i+\omega j}x_i+x_jx_{(\omega+1)i+\omega j}\quad\forall\, 
i,j\in\F_4.
\end{align}
This Nichols algebra was computed in \cite{grania1}; $\dim\B(4,\omega, -1)=72$. 

\subsubsection{The Nichols algebra $\B(5,2, -1)$}\label{subsubsec:nicholsF5}
The ideal $\cJ(5,2, -1)$ is generated by
$z_{(5, 2,-1)}:=(x_1x_0)^2+(x_0x_1)^2$ and
\begin{align}\label{eq:F5 2}
&x_i^2, \quad  x_ix_j+x_{-i+2j}x_i+x_{3i-2j}x_{-i+2j}+x_jx_{3i-2j} \quad\forall 
i,j\in\F_5.
\end{align}
This Nichols algebra was computed in \cite{AG2}; $\dim\B(5,2, -1)=1280$.

\subsubsection{The Nichols algebra $\B(5,3, -1)$}\label{subsubsec:nicholsF5 dual}
Since $\Aff(\F_5,3)$ is the dual rack of $\Aff(\F_5,2)$ and the $2$-cocycle is $-1$ we 
can apply Proposition \ref{prop:rel of Xdual q over G}. Then the ideal $\cJ(5,3, -1)$ 
is generated by
$z_{(5, 3,-1)}:=(x_1x_0)^2+(x_0x_1)^2$ and
\begin{align}\label{eq:F5 3}
&x_i^2, \quad  x_jx_i+x_ix_{-i+2j}+x_{-i+2j}x_{3i-2j}+x_{3i-2j}x_j\quad\forall i,j\in\F_5.
\end{align}

\subsubsection{The Nichols algebra $\B(7,3, -1)$}\label{subsubsec:nicholsF7}
The ideal $\cJ(7,3, -1)$  is generated by
$z_{(7,3,-1)}:=(x_2x_1x_0)^2+(x_1x_0x_2)^2+(x_0x_2x_1)^2$ and
\begin{align}\label{eq:F7 3}
&x_i^2, \quad x_ix_j+x_{-2i+3j}x_i+x_jx_{-2i+3j} \quad\forall i,j\in\F_7.
\end{align}
This Nichols algebra was computed in \cite{grania}; $\dim\B(7,3, -1)=326592$. 

\subsubsection{The Nichols algebra $\B(7,5, -1)$}\label{subsubsec:nicholsF7 dual}
As in \ref{subsubsec:nicholsF5 dual} we apply Proposition \ref{prop:rel of Xdual q over 
G} 
since $\Aff(\F_7^5)$ is the dual rack of $\Aff(\F_7,3)$. Then the ideal $\cJ(7,5, -1)$  is 
generated by
$z_{(7,5,-1)}:=x_2x_4x_0x_5x_3x_0+x_1x_3x_4x_5x_3x_2+x_0x_3x_6x_0x_4x_1$ and
\begin{align}\label{eq:F7 5}
&x_i^2, \quad x_jx_i+x_ix_{-2i+3j}+x_{-2i+3j}x_j\quad\forall i,j\in\F_7.
\end{align}

\subsubsection{The Nichols algebra $\B(4,\omega, \zeta)$}\label{subsubsec:nicholsA4}
Let $\xi\in \ku$ be a root of unity of order $3$. Then $\Aff(\F_4,\omega)$ admits the 
2-cocycle 
\begin{align}\label{eqn:cociclo-vendra}
 \zeta=(\zeta_{ij})_{i,j\in \F_4}=\left[\begin{smallmatrix}
                                \xi &\xi &\xi &\xi \\
                                \xi &\xi &-\xi &-\xi \\
                                \xi &-\xi &\xi &-\xi \\
                                \xi &-\xi &-\xi &\xi 
                               \end{smallmatrix}\right].
\end{align}
The Nichols algebra $\B(4,\omega, \zeta)$, see \cite[Proposition 7.9]{HLV}, has 
dimension
5184 and its ideal  of relations is generated by $x_0^3$, $x_1^3$, $x_\omega^3$, 
$x_{\omega^2}^3$, 
\begin{align*}
&\xi ^2\, x_0x_1 + \xi\, x_1x_\omega - x_\omega x_0,  && \xi ^2\, x_0x_\omega + \xi\, 
x_\omega x_{\omega^2} -
x_{\omega^2}x_0,\\
& \xi ^2\, x_1x_0 - \xi\, x_0x_{\omega^2} - x_{\omega^2}x_1,  && \xi ^2\, x_\omega x_1 + 
\xi\, x_1x_{\omega^2} +
x_{\omega^2}x_\omega,
\end{align*}
plus an extra degree six relation
\begin{multline*}
z_{(4,\omega, \zeta)}:=x_0^2 x_1x_\omega x_1^2 + x_0x_1x_\omega x_1^2 x_0 + x_1x_\omega 
x_1^2 x_0^2 +
x_\omega x_1^2 x_0^2 x_1 + x_1^2 x_0^2 x_1x_\omega \\+ x_1x_0^2 x_1x_\omega x_1
+x_1x_\omega x_1x_0^2 x_\omega + x_\omega x_1x_0x_1x_0x_\omega + x_\omega x_1^2 
x_0x_\omega x_0.
\end{multline*}

\subsection{About the top degree relation $z_{(b,N,q)}$}
In the rest of the section, the pair $(X,q)$ denotes one of the followings
\begin{align*}
&(\Aff(\F_3,2),-1),\ (\Aff(\F_4,\omega),-1),\ (\Aff(\F_5,2),-1),\ (\Aff(\F_5,3),-1),\\
&(\Aff(\F_7,3),-1),\ (\Aff(\F_7,5),-1)\ \mbox{or} \ (\Aff(\F_4,\omega),\zeta).
\end{align*}
We fix $n=2$ for the first six pairs and $n=3$ for the last one. We set
$\pi_n:T(X,q)\twoheadrightarrow\widehat{\B_n}(X,q)$ the natural projection.

Let $z=z_{(b,N,q)}$ be the top degree defining relation of $\B(X,q)$. Since $\cJ(X,q)$ is 
generated by $z$ and elements of degree $<\deg z$, $\ku\pi_n(z)\in\ydg$ via a central 
$t_z\in G$ and a multiplicative character $\chi_z:G\rightarrow\ku^*$, that is
\begin{align}\label{eqn:chiz}
\pi_n(z)\_{-1}\ot \pi_n(z)\_{0}=t_z\ot \pi_n(z)\quad\mbox{and}\quad 
g\cdot\pi_n(z)=\chi_z(g)\,\pi_n(z)
\end{align}
for all $g\in G$. Moreover, $\pi_n(z)$ is primitive in $\widehat{\B_n}(X,q)$ and therefore
\begin{align}\label{eq:z is skewprimi} 
\Delta(\pi_n(z))=\pi_n(z)\ot1+t_z\ot\pi_n(z)\quad\mbox{in}\quad\widehat{\B_n}(X,q)\#\ku 
G. 
\end{align}

\begin{lem}\label{le:f max degree rel}
For all $i\in X$, $\chi_{z}(g_i)=1$. If $q$ is constant, then $\chi_z=\chi_G^{\deg z}$.
\end{lem}

\begin{proof}
By Lemma \ref{le:the enveloping se proyecta} \eqref{ite:Autle:the enveloping se 
proyecta}, 
$G$ acts by rack automorphisms on $\Aff(\F_b,N)$. Let $\overline{t}$ be the automorphism defined
by $t\in G$. Let $K\subset G$ be the subgroup generated by $\{g_i:i\in X\}$ and 
$\mathcal{Z}(K)$ be its center. By \cite[Lemma 1.9 (2)]{AG2} and \eqref{eq:inn equal 
aut}, 
$K/\mathcal{Z}(K)=\Inn_\rhd\Aff(\F_b,N)=\F_b\rtimes_{N}C_{\ord N}=\Aut_\rhd\Aff(\F_b,N)$. Thus 
there is a multiplicative character
$\lambda:\Aut_\rhd\Aff(\F_b,N)\rightarrow\ku^*$ such that
$$
\chi_z(t)\pi_n(z)=t\cdot\pi_n(z)=\widetilde\chi_z(t)\lambda(\overline{t})\,\pi_n(z),
$$
where $\widetilde\chi_z$ is given by the 1-cocycle $\{\chi_i\}_{i\in X}$. If $q$ is 
constant, then $\widetilde\chi_z=\chi_G^{\deg z}$. If $q$ is not constant, then it is 
easy to 
check that $\widetilde\chi_{z|K}=\varepsilon$. Therefore, to finish we have to prove that 
$\lambda=\varepsilon$. Let $M=\ku G\cdot z\subset T(X,q)$.

{\it \underline {Case}} $\Aff(\F_4,\omega)$. Let
$\cO=\{(\omega\,1\,0),(1\,\omega\,\omega^2),(0\,1\,\omega^2),(0\,\omega^2\,\omega)\}
\subset
\F_4^3$ and
$$X_{(abc)}:=(x_ax_bx_c)^2+(x_bx_cx_a)^2+(x_cx_ax_b)^2,\quad(a\, b\, c)\in\cO.$$
Then $z=X_{(\omega10)}$ and $M=\ku\{X_\sigma\}_{\sigma\in\cO}$. Let 
$Y=\sum_{\sigma\in\cO}X_\sigma$ and $C=\ku\{
X_\sigma-X_\tau:\sigma,\tau\in\cO\}$. Then $M=C\oplus\ku Y$ is a sum of simple
$\F_4\rtimes_{\omega}C_{3}$-submodules. Thus $\pi_2(C)=0$ and $\pi_2(M)\simeq\k Y$ as
$\F_4\rtimes_{\omega}C_{3}$-modules.

{\it \underline {Case}} $\Aff(\F_5,2)$. Here $z=(x_1x_0)^2+(x_0x_1)^2$. By \eqref{eq:F5 2}, it
holds that
\begin{align*}
\pi_2((x_0x_2)^2)&=\pi_2(-x_0(x_{3}x_2+x_{1}x_{3}+x_0x_{1})x_2)=\pi_2(-x_0x_{1}x_{3}
x_2)\quad\mbox{ and}\\
\pi_2((x_2x_0)^2)&=
\pi_2((x_{1}x_0)^2+(x_0x_{1
})^2+x_0x_{1}x_{3}x_2).
\end{align*}
Hence $\pi_2((x_{2}x_{0})^2+(x_{0}x_{2})^2)=\pi_2(z)$ and thus $(0\times
t)\cdot\pi_2(z)=\pi_2(z)$.
Since $(0\times t)(1\times1)=(2\times1)(0\times t)$ in
$\F_5\rtimes_{2}C_{4}=\langle(0\times t),(1\times1)\rangle$,
$\pi_2(M)$ is the trivial $\F_5\rtimes_{2}C_{4}$-module.

{\it \underline {Case}} $\Aff(\F_5,3)$. As in \ref{subsec:dual rack}, we denote by 
$\{y_i\}_{i\in\F_5}$ the basis of $V(5,3, -1)$ and recall that 
$\fc(z_{(5,2,-1)})=z_{(5,3,-1)}$. 
Let $\vartheta:(\widehat{\B_2}(5,2,-1)\#\ku 
G)^{\cop}\rightarrow\widehat{\B_2}(5,3,-1)\#\ku G$ be the Hopf algebra map given by 
$\vartheta(x_i\#g)=y_i\#g_ig$ for all $i\in X,$ $g\in G$. 
Then 
$\vartheta(\overline{g_i}\cdot\pi_2(z_{(5,2,-1)}))=\overline{g_i}\cdot\pi_2(z_{(5,3,-1)}
)\#\,t_{z_{(5,2,-1)}}$ since the action is induced by the adjoint action. Hence 
$\lambda=\varepsilon$ because before we proved that
$$
\vartheta(\overline{g_i}\cdot\pi_2(z_{(5,2,-1)}))=\vartheta(\pi_2(z_{(5,2,-1)}))=\pi_2(z_{
(5,3,-1)})\#\,t_{z_{(5,2,-1)}}. 
$$

{\it \underline {Cases}} $\Aff(\F_7,3)$ and $(\Aff(\F_4,\omega),\zeta)$. In both cases, $
(0\times t)\cdot\pi_2(z)=\pi_2(z)$,  using \cite{GAP, GBNP}. Then we proceed as for 
 $\Aff(\F_5,2)$.

{\it \underline {Case}} $\Aff(\F_7,5)$ is similar to $\Aff(\F_5,3)$ since $\Aff(\F_7,5)^*\simeq\Aff(\F_7,3)$.
\end{proof}

In the following, $\widehat{\B_n}(X,q)\#\ku G$ is a right $\B(X,q)\#\ku G$-comodule via 
the 
natural projection.

\begin{lem}\label{le:z is central}
It holds that $\pi_n(z)$ is central in $\widehat{\B_n}(X,q)$ and the subalgebra of right 
$\B(X,q)\#\ku G$-coinvariants is the polynomial algebra
$\ku[\pi_n(z)t_z^{-1}]$.
\end{lem}
\begin{proof}
We check that $\pi_n(z)$ is central using \cite{GAP} together with the package 
\cite{GBNP} 
in all the cases except for $\Aff(\F_5,3)$ and $\Aff(\F_7,2)$. For these we keep the notation of 
the previous proof and proceed as follows. If $i\in\F_b$,
\begin{align*}
0&=\vartheta(x_i\pi_2(z_{(b,N,-1)})-\pi_2(z_{(b,N,-1)})x_i)\\
&=(y_i\# 
g_i)(\pi_2(z_{(b,N^{-1},-1)})\#\,t_{z_{(b,N,-1)}})-(\pi_2(z_{(b,N^{-1},-1)})\#\,t_{z_{(b,N
,-1)}})(y_i\#g_i)\\
&=(y_i\pi_2(z_{(b,N^{-1},-1)})-\pi_2(z_{(b,N^{-1},-1)})y_i)\#\,t_{z_{(b,N,-1)}}g_i
\end{align*}
here we use the above lemma and that $t_{z_{(b,N,-1)}}$ is central. Thus the first part 
of 
the lemma is proved. Then $\pi_2(z)t_z^{-1}$ generates a normal subalgebra which forms 
the 
coinvariants by \cite[Remark 5.5]{AAnMGV}. It is a polynomial algebra by \cite[Lemma 
5.13]{AAnMGV}.
\end{proof}

\subsection{The Nichols algebras $\B(q,b,N)\in\ydgdual$}\label{subsec:nichols afin sobre el dual}
For each $(X,q)$ as above, consider the object $W(q,X)\in \ydgdual$ as in
\eqref{eqn:yetter-drinfeld-dual}. From now on, $T(q,b,N)$ and $\B(q,b,N)$ denote 
respectively its tensor algebra and the Nichols algebra with ideal of relations 
$\cJ(q,b,N)$. Let $\pi_n:T(q,X)\twoheadrightarrow\widehat{\B_n}(q,X)$ be the natural 
projection.

\begin{pro}\label{pro:rels-dual}
 \begin{enumerate}\renewcommand{\theenumi}{\alph{enumi}}\renewcommand{\labelenumi}{
(\theenumi)}
  \item The ideal $\cJ(-1,3,2)$ is generated by \eqref{eq:F3 2}.
\item The ideal $\cJ(-1,4,\omega)$ is generated by
\begin{align}\label{eq:F4 omega prima}
\tag{\ref{eq:F4 omega}'}x_i^2,& \quad x_jx_i+x_ix_{(\omega+1)i+\omega
j}+x_{(\omega+1)i+\omega j}x_j,\quad i,j\in\F_4
\end{align}
and $z'_{(-1,4, \omega)}:=(x_\omega
x_{\omega^2}x_0)^2+(x_1 x_{\omega^2}x_\omega)^2+(x_0 x_{\omega^2}x_1)^2$.
\item  The ideal $\cJ(-1,5,2)$ is generated by
\begin{align}\label{eq:F5 2 prima}
\tag{\ref{eq:F5 2}'}&x_i^2,\quad
x_jx_i+x_ix_{-i+2j}+x_{-i+2j}x_{3i-2j}+x_{3i-2j}x_j,\quad i,j\in\F_5
\end{align}
and $z'_{(-1,5, 2)}:=x_0x_2x_3x_1+x_1x_4x_3x_0$.
\item  The ideal $\cJ(-1,5,3)$ is generated by
\begin{align}\label{eq:F5 3 prima}
\tag{\ref{eq:F5 3}'}&x_i^2,\quad
x_ix_j+x_{-i+2j}x_i+x_{3i-2j}x_{-i+2j}+x_jx_{3i-2j},\quad i,j\in\F_5
\end{align}
and $z'_{(-1,5,3)}:=x_1x_3x_2x_0+x_0x_3x_4x_1$.
\item The ideal $\cJ(-1,7,3)$ is generated by
\begin{align}\label{eq:F7 3 prima}
\tag{\ref{eq:F7 3}'}&x_i^2, \quad x_jx_i+x_ix_{-2i+3j}+x_{-2i+3j}x_j, \quad i,j\in\F_7
\end{align}
and $z'_{(-1,7,3)}:=x_2x_6x_4x_2x_5x_0+x_1x_5x_2x_3x_6x_2+x_0x_6x_4x_5x_6x_1$.
\item The ideal $\cJ(-1,7,5)$ is generated by
\begin{align}\label{eq:F7 5 prima}
\tag{\ref{eq:F7 5}'}&x_i^2, \quad x_ix_j+x_{-2i+3j}x_i+x_jx_{-2i+3j}, \quad i,j\in\F_7
\end{align}
and $z'_{(-1,7,5)}:=x_0x_5x_2x_4x_6x_2+x_2x_6x_3x_2x_5x_1+x_1x_6x_5x_4x_6x_0$.
\item The ideal $\cJ(\zeta,4,\omega)$ is generated by $x_0^3$, $x_1^3$, 
$x_\omega^3$, 
$x_{\omega^2}^3$,
\begin{align*}
&\xi\, x_\omega x_1 -x_0x_\omega - \xi^2x_1x_0,\quad  \xi\, x_{\omega^2}x_\omega - 
x_0x_{\omega^2} -\xi^2x_\omega x_0,\\
&\xi\, x_{\omega^2}x_0 - x_1x_{\omega^2} +\xi^2 x_0x_1,\quad \xi\, x_{\omega^2}x_1 + 
x_\omega x_{\omega^2} + \xi^2x_1x_\omega\,\mbox{ and}\\
z'_{(\zeta,4,\omega)}:=&\,x_0^2 x_{\omega^2}x_0x_1^2 + x_0x_\omega x_0 x_{\omega^2}^2 
x_0 
+ x_1x_0 x_\omega^2 x_0^2 +
x_\omega x_{\omega^2}^2 x_\omega^2 x_1 - x_1^2 x_\omega^2 x_0x_\omega \\
+ x_1&x_{\omega^2}^2 x_\omega x_{\omega^2} x_1+x_1x_0 x_\omega x_{\omega^2}^2 x_\omega - 
x_\omega x_{\omega^2} x_0x_1x_{\omega^2}x_\omega + x_\omega x_{\omega^2}^2 x_\omega x_1 
x_0.
\end{align*}
\end{enumerate}
\end{pro}
\pf
In (a), (b), (c), (e) and (g) we apply the functor $(F,\eta)$ given by
\eqref{prop:equiv de cat gdual} and use Lemma \ref{thm:generators of nichols algebras
via functors}. In (d), resp. (f), we apply Proposition \ref{prop:rel of Xdual q over 
kalaG} since it corresponds to the dual case of (c), resp. (e), and the $2$-cocycle is 
$-1$.
\epf

Set $z'=z'_{(q,b,N)}=\eta_{\deg z_{(b,N,q)}}^{-1}(z_{(b,N,q)})$. Then 
$\ku\pi_n(z')\in\ydgdual$ as follows 
\begin{align}\label{eqn:chiz dual}
\pi_n(z')\_{-1}\ot\pi_n(z')\_{0}=\chi_z^{-1}\ot\pi_n(z')\quad\mbox{and}\quad 
\delta_g\cdot\pi_n(z)=\delta_{g,t_z^{-1}}\,\pi_n(z)
\end{align}
for all $g\in G$ by Lemma \ref{lem:generators of nichols algebras via functors}
\eqref{eq:lem:generators of nichols algebras via functors cocientes} and Lemma \ref{le:f
max degree rel}. Also, $\pi_n(z')$ is primitive in $\widehat{\B_n}(q,X)$ and therefore
\begin{align}\label{eq:z is skewprimi dual} 
\Delta(\pi_n(z'))=\pi_n(z')\ot1+\chi_z^{-1}\ot\pi_n(z')\quad\mbox{in}\quad\widehat{\B_n}(q,X)\#\ku^G. 
\end{align}

In the following, $\widehat{\B_n}(q,X)\#\ku^G$ is a right $\B(q,X)\#\ku^G$-comodule via 
the 
natural projection. 
\begin{lem}\label{le:f max degree rel prima}
It holds that $\pi_n(z')$ is central in $\widehat{\B_n}(q,X)$ and the subalgebra of right 
$\B(q,X)\#\ku^G$-coinvariants is the polynomial algebra
$\ku[\pi_n(z')\chi_z]$.
\end{lem}

\begin{proof}
If $i\in\F_b$, then $\pi_n\eta_7^{-1}(x_iz-zx_i)=0$ by Lemma \ref{le:z is central} and 
Lemma \ref{thm:generators of nichols algebras via functors} \eqref{eq:pre generatos thm:generators of nichols algebras via functors}. By 
\eqref{eq:asociatividad de los etas},
$\eta_7^{-1}(x_iz-zx_i)=x_{t_z^{-1}\cdot i}z'-z'x_i=x_{i}z'-z'x_i$, here we use Lemma 
\ref{le:f max degree rel} and that $t_z$ is central.
Hence $\pi_n(z')$ is central in $\widehat{\B_n}(q,X)$. The lemma follows using 
\cite[Remark 5.5, Lemma 5.13]{AAnMGV} as in Lemma \ref{le:z is central}. 
\end{proof}

\section{Lifting via cocycle deformation}\label{very general results} 

Let $H$ be a semisimple Hopf algebra and $V\in\ydh$, $\dim V<\infty$. Assume that 
the ideal $\cJ(V)$ defining 
the Nichols algebra $\B(V)$ is finitely generated and let $\mG$ be a minimal set 
of homogeneous generators of $\cJ(V)$. In \cite{AAnMGV} a strategy was developed to 
compute all the liftings of $\B(V)$ 
over $H$ as cocycle deformations of $\B(V)\#H$. We briefly recall this 
strategy, see \cite[Section 5]{AAnMGV} for details.

Set $\mT(V)=T(V)\# H$ and $\mH=\B(V)\# H$. Let $\mG =  \mG_0 \cup \mG_1 \cup\dots \cup 
\mG_{N}$ be an {\it adapted stratification} of $\mG$ \cite[5.1]{AAnMGV}. Among other things, this ensures 
that
$$\B_k= \B_{k-1}/\langle \mG_{k-1}\rangle,\quad1\leq k\leq N+1,$$
are braided Hopf algebras in $\ydh$ where $\B_0=T(V)$. Then we have a chain of subsequent 
quotients of Hopf algebras
$$\mT(V)\twoheadrightarrow\B_1\# H\twoheadrightarrow\dots\twoheadrightarrow
\B_N\# H\twoheadrightarrow \mH=\B_{N+1}\# H.$$
The Strategy basically consists in the following two steps:

\begin{itemize}
 \item[(1)] To compute at each level a family of cleft objects of $\B_k\# H$ as 
quotients of 
cleft objects of $\B_{k-1}\# H$, following the results in \cite{Gu}. 
\end{itemize}
To do this, we start with the trivial
cleft object for $\mT(V)$. In the final level, we have a set $\Lambda$ of cleft objects
of $\mH$ and hence a list of cocycle deformations
$L$, which arise as $L\simeq L(\mA,\mH)$, for $\mA\in\Lambda$ as in \cite{S}. 
\begin{itemize}
 \item[(2)] To check that any lifting of $\B(V)$ over $H$ is obtained as one of
these deformations.
\end{itemize}
In {\it loc.~ cit.} a series of tools to deduce this was developed. In particular, it was 
studied in \cite[Section 4]{AAnMGV} the shape of all the possible liftings. We refine the 
results there for copointed liftings in Subsection 
\ref{general results}. 

\smallbreak

We use the Strategy to prove the main theorems. In that order, we carry out the 
Strategy in the next subsection under certain general conditions which are satisfied in 
our case.

\subsection{Pointed Lifting of Nichols algebras with a single top degree 
relation}\label{general results groups} 

Let $X$ be an indecomposable rack with a 2-cocycle $q$. Let $G$ be a finite group
and
$(\cdot,g,\{\chi_i\}_{i\in X})$ be a principal YD-realization of $(X,q)$. Let
$V=\ku\{x_i\}_{i\in X}$ be the
corresponding Yetter-Drinfeld module over $G$, see \eqref{eqn:yetter-drinfeld}.
Assume that the Nichols algebra $\B(V)$ is finite dimensional. 

Let $n\in\N$ be such that $\ord q_{ii}=n\geq2$. Then $x_i^n\in\cJ(V)$ for all $i\in X$.

\smallbreak

Recall from \cite{GG} that the space of quadratic relations in $\cJ(V)$ is spanned by
$\{b_C\}_{C\in \mR'}$ where $\mR'$ is a subset of the set $\mR=X\times X/\sim$ of
classes of the equivalence relation generated by $(i,j)\sim(i\rhd j,i)$. More precisely,
$C=\{(i_2,i_1), \ldots, (i_{n(C)},i_1)\}\in\mR'$ iff $\prod_{h=1}^{n(C)} 
q_{i_{h+1},i_h}=(-1)^{n(C)}$ and then
\begin{align}\label{eqn:bc}
 &b_C:= \sum_{h=1}^{n(C)}\eta_h(C)\, x_{i_{h+1}}x_{i_h},
\end{align}
where $\eta_1(C)=1$ and $\eta_h(C)=(-1)^{h+1}q_{{i_2i_1}}q_{{i_3i_2}}\ldots 
q_{{i_hi_{h-1}}}$, $h\ge 2$.

\smallbreak

Set $\mT(V)=T(V)\#\ku G$ and $\pi_n:T(V)\to \widehat{\B_n}(V)$. 
We assume that there is a generator $z\in\cJ(V)$ with $\deg z>n$ such that
\begin{itemize}
\item $\ku \pi_n(z)\in\ydg$, that is, there exist a central  $t_z\in G$ and a 
multiplicative character $\chi_z:G\rightarrow\ku^*$ such that \eqref{eqn:chiz} holds; 
\item $\pi_n(z)$ is primitive in $\widehat{\B_n}(V)$ and hence  \eqref{eq:z is 
skewprimi} is satisfied in $\widehat{\B_n}(V)\#\ku G$;
\item the following holds in $\widehat{\B_n}(V)\#\ku G$:
\begin{align}
\label{eqn:znormal1} x_i\pi_n(z)= \pi_n(z)x_i, \ i\in X \quad \text{ and } \quad   
t_z\pi_n(z)= \pi_n(z)t_z.
\end{align}
\end{itemize}

We assume that the ideal $\cJ(V)$ admits an adapted stratification:
\begin{align}\label{eqn:stratification}
\mG_0&=\{x_i^n:\, i\in X\}, \, \mG_1=\{b_C:\, C\in\mR'\}\setminus \{x_i^2:\, i\in
X\}, \, \mG_2=\{z\, \}
\end{align}
and apply the Strategy in this setting. Set $\mH_i=\B_{i-1}/\langle\mG_{i-1}\rangle\#\ku 
G$ for $i=1,2,3$ with $\B_0=T(V)$. We also assume that
\begin{align}
\label{eqn:gin} \qquad g_i^n&\neq g_j \quad \text{ and } \quad  g_k \neq g_ig_j, && \text{
for every } i,j,k\in
X,&\\
\label{eqn:g(z)}\qquad  t_z&\neq g_i, &&  \text{ for every } i\in
X.&
\end{align}
Notice that \eqref{eqn:gin} is not a relevant restriction by Lemma \ref{le:neq in group 
rack}. In particular, this lemma applies to affine racks.

We shall consider scalars $\lambda_1, \lambda_2, \lambda_3\in\ku$ subject to the following
conditions
\begin{align}
\label{eqn:cond1-A} &\lambda_1=0 \, \text{ if }\, \chi_i^n\neq \eps\, \forall\, i, &&
\lambda_2=0 \, \text{ if }\,  \chi_i\chi_j\neq \eps\,  \forall\, i,j, &&\lambda_3=0 \,
\text{ if } \, \chi_z \neq \eps.
\end{align}

Let $\lambda_1,\lambda_2\in\ku$ subject to \eqref{eqn:cond1-A} and let $b_C$ be as in
\eqref{eqn:bc}. Set
\begin{align}
\label{a} \mA_1(\lambda_1) &= \mT(V)/\langle x_i^n-\lambda_1 :\, i\in X\rangle,\\
\label{eqn:galoisnonula} \mA_2(\lambda_1,\lambda_2)&= \mA_1(\lambda_1)/\langle b_C -
\lambda_2\,:\, C\in\mR'\rangle.
\end{align}
Note that $\mA_1(\lambda_1)\neq 0$. In fact, $\mA_1(\lambda_1)\simeq T(V)/\langle 
x_i^n-\lambda_1 :\, i\in X\rangle\ot\ku G$ as vector spaces by the choice of 
$\lambda_1$ in \eqref{eqn:cond1-A} and we can define a nonzero 
algebra map $F:T(V)/\langle x_i^n-\lambda_1 :\, i\in X\rangle\rightarrow\ku$ by 
$F(x_i)=\lambda_1^{1/n}$ for all $i\in X$.

\smallbreak

Set also $\mL_1(\lambda_1)=\mT(V)/\langle x_i^n-\lambda_1(1- g_i^n):\, i\in
X\rangle$ and 
\begin{align*}
 \mL_2(\lambda_1,\lambda_2)&=\mL_1(\lambda_1)/\langle b_C - \lambda_2(1 - g_ig_j):\,
C\in\mR',\,
(i,j)\in C \rangle.
\end{align*}
It is straightforward to see that $\mL_k$ is a Hopf algebra quotient of $\mT(V)$ and
$\mA_k$ is naturally a $(\mL_k, \mH_k)$-bicomodule algebra with coactions $\delta_L^k$, 
$\delta_R^k$ induced by the comultiplication in $\mT(V)$.

\begin{pro}\label{pro:lifting1}Let $\mA_1=\mA_1(\lambda_1)$,
$\mA_2=\mA_2(\lambda_1,\lambda_2)$, $\mL_1=\mL_1(\lambda_1)$  and
$\mL_2=\mL_2(\lambda_1,\lambda_2)$. Assume $\mA_2\neq 0$. Then
\begin{enumerate}\renewcommand{\theenumi}{\alph{enumi}}\renewcommand{\labelenumi}{
(\theenumi)}
\item $\mA_k$ is a right Galois object of $\mH_k$. 
\item There is a section $\gamma_k:\mH_k\to \mA_k$ with ${\gamma_k}_{|\ku G}=\id$.
\item $L(\mA_k,\mH_k)\cong \mL_k$. Hence $\mL_k$ is a cocycle deformation of $\mH_k$.
\end{enumerate}
\end{pro}
\pf
(a) follows by \cite[Theorem 8]{Gu} applied to a suitable right coideal subalgebra $Y_k$.
If $k=1$, we take $Y_1$ to be generated by $x_i^ng_i^{-n}$ for some $i\in X$. If $k=2$,
this is done in several steps, one for each $C\in\mR'$, up to conjugacy, taking $Y_{2,C}$ 
as the subalgebra generated by $b_Cg_j^{-1}g_i^{-1}$ for $(i,j)\in C$.

(b) This is \cite[Lemma 5.8 (b)]{AAnMGV}.

(c) follows by applying \cite[Proposition 5.10]{AAnMGV}.
\epf

It is possible to use \cite{GAP, GBNP} in specific examples to check that $\mA_2(\lambda_1,\lambda_2)\neq0$. 
We 
do this in the next section to prove Main Theorem  1. We now compute Galois objects 
for $\mH=\mH_3=\B(V)\#\ku G$.

\begin{pro}\label{pro:liftings3}
Assume that $\mA_2(\lambda_1,\lambda_2)\neq 0$ for some $\lambda_1,\lambda_2$. 
\begin{enumerate}\renewcommand{\theenumi}{\alph{enumi}}\renewcommand{\labelenumi}{
(\theenumi)}
\item There exists $a_X\in\mA_2(\lambda_1,\lambda_2)$ and $\lambda_3\in\ku$ subject to
\eqref{eqn:cond1-A} such
that
\begin{align*}
\mA=\mA(\lambda_1,\lambda_2,\lambda_3)=\mA_2(\lambda_1,\lambda_2)/\langle z - a_X - 
\lambda_3 
\rangle. 
\end{align*}
is a Galois object of $\mH_3$. 
\item  $
L(\mA,\mH_3)\cong \mL_3(\lambda_1,\lambda_2,\lambda_3)$ where
$$
\mL_3(\lambda_1,\lambda_2,\lambda_3) = \mL_2(\lambda_1,\lambda_2)/\langle
z-s_X-\lambda_3(1-t_z)\rangle
$$
and $s_X\in\mL_2(\lambda_1,\lambda_2)$ is such that
\begin{align}\label{eqn:zdeformed}
 (z-s_X)\ot 1=\delta_L(\gamma(z))-t_z\ot \gamma(z).
\end{align}
\end{enumerate}
\end{pro}
\pf 
Set $\mH'=\mH_2$, $\mA'=\mA_2(\lambda_1,\lambda_2)$, $\mL'=\mL_2(\lambda_1, \lambda_2)$, 
$\gamma=\gamma_2:\mH'\to \mA'$. We consider $\mH'$ as a right $\mH$-comodule via the 
natural projection.
We use \cite[Theorem 4]{Gu} to find cleft objects of $\mH$. For that, we have to 
compute the subalgebra $\mH'^{\co\mH}$ of right $\mH$-coinvariants and the set 
$\Alg_{\mH'}^{\mH'}(\mH'^{\co\mH}, \mA')$ of algebra maps from $\mH'^{\co\mH}$ to 
$\mA'$ in $\mathcal{YD}_{\mH'}^{\mH'}$.

Let $Y$ be the subalgebra of $\mH'$ generated by $zt_z^{-1}$. Then $Y$ is normal, by 
\eqref{eqn:znormal1}, and a polynomial algebra, by \cite[Lemma 5.13]{AAnMGV} and \eqref{eqn:znormal1}.  Hence 
$Y=\mH'^{\co\mH}$ by \cite[Remark 5.5]{AAnMGV}. By \cite[Remark 5.11]{AAnMGV} $f\in 
\Alg_{\mH'}^{\mH'}(Y, \mA')$ if and only if $f(zt_z^{-1})=\gamma(zt_z^{-1})-\lambda_3 
t_z^{-1}$
for some $\lambda_3\in\ku$.

\smallbreak

Therefore (a) follows with $a_X=z-\gamma(z)$ by \cite[Theorem 4]{Gu}. Now (b) follows by 
\cite[Corollary 5.12]{AAnMGV}.
\epf

We need to compute $\gamma_2(z)$, $\delta^2_R(z)$ and $\delta^2_L\big(\gamma_2(z)\big)$ 
to apply the above proposition. We explain in Appendix how we can do this using 
\cite{GAP, GBNP}.

\smallbreak

The pointed liftings of $\B(V)$ are given by the next theorem.

\begin{theorem}\label{thm:all-liftings}
Let $L$ be a lifting of $\B(V)$ over $\ku G$. Then there exist scalars 
$\lambda_1,\lambda_2,\lambda_3\in\ku$ such that $L\cong 
\mL_3(\lambda_1,\lambda_2,\lambda_3)$ and hence $L$ is a cocycle deformation of 
$\B(V)\#\ku G$.
\end{theorem}
\pf
Consider the lifting map $\phi:\mT(V)\to L$ defined by \eqref{eq:properties of A and 
phi}. 
If $r\in \mG_0\cup \mG_1$, then $r$ is a skew-primitive in $\mT(V)$ and thus $\phi(r)\in 
L_{[1]}$. Moreover, $\phi(r)\in L_{[0]}$ by \eqref{eqn:gin}. Hence $\phi$ induces
$\phi':\mL_2(\lambda_1,\lambda_2)\twoheadrightarrow L$ for some $\lambda_1,\lambda_2\in 
\ku$. 

 It follows that $\overline{z}=z-s_X$ is a $(1,t_z)$-primitive in 
$\mL_2(\lambda_1,\lambda_2)$ and thus
$\phi'(\overline{z})\in L_{[1]}$. By \eqref{eqn:g(z)} we see that $\phi'(\overline{z})\in
L_{[0]}$ and therefore there is
$\lambda_3\in\ku$ such that $\phi'(\overline{z})=\lambda_3(1-t_z)$. Therefore $\phi'$
induces $\phi'':\mL(\lambda_1,\lambda_2,\lambda_3)\twoheadrightarrow
L$ and this is an isomorphism since both algebras have dimension $\dim\B(V)|G|$.
\epf

To avoid repetitions, we further normalize the scalars $\lambda_1,\lambda_2,\lambda_3$ by
\begin{align}
\label{eqn:cond1-L} &\lambda_1=0\, \text{ if }\,  g_i^n=1, && \lambda_2=0\, \text{ if
} \, g_ig_j=1, && \lambda_3=0\,  \text{ if }\, t_z=1,
\end{align}
and consider the set 
\begin{align}\label{eqn:Sx}
  \Ss_X= \{(\lambda_1,\lambda_2,\lambda_3)\in\ku^3|\, \text{satisfying
\eqref{eqn:cond1-A} and \eqref{eqn:cond1-L}}\}.
\end{align}
\begin{pro}\label{pro:isoclases}
If $(\lambda_1,\lambda_2,\lambda_3), (\lambda_1',\lambda_2',\lambda_3')\in \Ss_X$ then
$\mL_3(\lambda_1,\lambda_2,\lambda_3)\cong\mL_3(\lambda_1',\lambda_2',\lambda_3')$ if and
only if $(\lambda_1,\lambda_2,\lambda_3)=\mu(\lambda_1',\lambda_2',\lambda_3')$ for some $\mu\in\ku$.
\end{pro}
\pf
Follows as \cite[Lemma 6.1]{GG}.
\epf
The results above restrict to the case in which there is no relation $z$ as
in \eqref{eqn:stratification}, that is  when $\cJ(V)$ admits an adapted
stratification $\mG_0\cup\mG_1$. We collect this information in the following
corollary. In this case we also denote
\begin{align}\label{eqn:Sxprima}\tag{\ref{eqn:Sx}'}
  \Ss_X= \{(\lambda_1,\lambda_2)\in\ku^2|\, \text{satisfying
\eqref{eqn:cond1-A} and
\eqref{eqn:cond1-L}}\}.
\end{align}
\begin{cor}\label{cor:all-liftings}
Let $\cJ(V)$ be as above. Let $L$ be a lifting of $\B(V)$ over $\ku G$.
\begin{enumerate}\renewcommand{\theenumi}{\alph{enumi}}\renewcommand{\labelenumi}{
(\theenumi)}
\item There exist $(\lambda_1,\lambda_2)\in\Ss_X$ such that
$L\cong\mL_2(\lambda_1,\lambda_2)$.
\item If $(\lambda_1,\lambda_2), (\lambda_1',\lambda_2')\in \Ss_X$, 
then $\mL_2(\lambda_1,\lambda_2)\cong \mL_2(\lambda_1',\lambda_2')$ if
and only if $(\lambda_1,\lambda_2)=\mu(\lambda_1',\lambda_2')$ for some $\mu\in\ku$..
\item If $\mA_2(\lambda_1,\lambda_2)\neq 0$, then $L$ is a cocycle deformation of
$\B(V)\#\ku G$.$\hfill\square$
\end{enumerate}
\end{cor}

\subsection{The shape of copointed liftings}\label{general
results} 
Let $G$ be a finite group and $V\in\ydg$, $\dim V<\infty$. If
$\{v_i\}$, $\{\alpha_i\}$ are dual bases of $V$ and $V^*$, set $f_{ji}: \ku
G\rightarrow\ku$, $h\mapsto\langle\alpha_j,h\cdot v_i\rangle$. By
\eqref{prop:equiv de cat gdual}, $V\in\ydgdual$ via
\begin{align}\label{eq:dual-yd}
f\cdot v&=\langle \cS(f),v_{(-1)}\rangle v_{(0)}\quad\mbox{and}\quad 
\lambda(v_i)=\sum_{j}\cS^{-1}(f_{ji})\ot v_j
\end{align}
for all $f\in \ku^G$, $v\in V$. This definition is independent of the basis $\{v_i\}$. We
say that $\{e_{ij}:=\cS^{-1}(f_{ji})\}$ is the
{\it comatrix basis associated to} $V$ and $\{v_i\}$.

In particular, let $\cdot$ be an action of $G$
on a set $X$ and let $\{\chi_i: G\to \ku\}_{i\in X}$ be a 1-cocycle, see page
\pageref{item:1-cocycle}. Then $\ku X$ with basis $\{m_i\}_{i\in X}$ is
a $G$-module via
\begin{align}\label{eq: Gmod assoc to X chi}
g\cdot m_i=\chi_i(g) m_{g\cdot i}\quad\mbox{for all}\quad i\in X, g\in G
\end{align}
and the comatrix basis $\{e_{ij}\}$ associated to $\ku X$ and $\{m_i\}_{i\in X}$
is
\begin{align*}
e_{ij}=\sum_{g\in G}\delta_{j,g\cdot i}\,\chi_i(g)\delta_{g^{-1}}\quad\mbox{for all
}i,j\in
X.
\end{align*}

Let $A$ be a lifting of $\B(V)$ over $\ku^G$ with a
lifting map $\phi:T(V)\#\ku^G\rightarrow A$, recall \eqref{eq:properties of A and phi}. 
We 
consider the first term of the coradical filtration $A_{[1]}\in\ydgdual$ in such a way 
that 
$\phi_{|(\ku\oplus V)\#\ku^{G}}:(\ku\oplus V)\#\ku^{G}\rightarrow A_{[1]}$ is an 
isomorphism 
in $\ydgdual$, cf. \cite[Section 4]{AAnMGV}. Then we identify both modules.

The following lemma is a particular case of \cite[Lemma 4.8]{AAnMGV}. It helps us to
describe the image by $\phi$ of a submodule $M$ of $T(V)$ in $\ydgdual$ {\it compatible 
with} $\phi$ \cite[Definition 4.7]{AAnMGV}, that is
$$\D(\phi(m))=\phi(m)\ot1+m_{(-1)}\ot\phi(m_{(0)})\mbox{ for all }m\in M.$$
Then $\phi(m)\in (\ku\oplus V)\#\ku^G$. We define the ideal of $T(V)$
\begin{align}\label{eq:IM}
\cI_M=\langle m-\phi(m):m\in M\rangle. 
\end{align}

Note that if $M\in\ydgdual$, then $M[e]$ and
$M^\times$ are submodules of $M$ in $\ydgdual$ such that 
$M=M[e]\oplus M^\times$. In fact, $\ydgdual$ is a semisimple category and the supports of 
the simple objects are conjugacy classes of $G$ \cite[Proposition 3.1.2]{AG1}.

\begin{lem}\label{thm:liftings}
Let $G$, $V$, $A$ and $\phi$ be as above. Let $M\subset T(V)$ be compatible with $\phi$
and $\{e_{ij}\}$ be the comatrix basis associated to $M[e]$ and $\{m_i\}_{i=1}^r$. Then
\begin{enumerate}\renewcommand{\theenumi}{\alph{enumi}}\renewcommand{\labelenumi}{
(\theenumi)}
\item\label{item:Mtimes} $\phi_{|M^\times}:M^\times\rightarrow V$ is a morphism in
$\ydgdual$.
\smallbreak
\item\label{cor:liftings Mtimes} Assume that $M=M^\times$ is a simple object in
$\ydgdual$ and $V\simeq M^m\oplus P$ with $m$ maximum.
Then there exist $\lambda_1, \dots, \lambda_m\in\ku$ such that
$$\phi_{|M}=\lambda_1\id_M\oplus\cdots\oplus\lambda_m\id_M\oplus \, 0.$$
In particular, $\phi_{|M}=0$ if $\supp M\cap\supp V=\emptyset$.
\smallbreak 
\item\label{item:Me} If $e\notin\supp V$, then there exist $a_1, \dots, a_r\in\ku$ such
that 
$$\phi(m_i)=a_i-\sum_{j=1}^ra_{j}c_{ij}\mbox{ for all }i=1, \dots, r.$$
\smallbreak
\item\label{cor:liftings Me}
If $e\notin\supp V$ and $M=M[e]$ with the $G$-action on $M$ satisfying \eqref{eq: Gmod
assoc to X chi}, then there exist
$(a_{i})_{i\in X}\in\ku^X$ such that
$$
\phi(m_i)=\sum_{g\in G}(a_{i}-\chi_i(g)a_{g\cdot i})\delta_{g^{-1}}\mbox{ for all }i\in X.
$$
\smallbreak
\item\label{item:iso between As} Let $\phi':T(V)\#\ku^G\rightarrow A'$ be a lifting map and $\Theta:A\rightarrow A'$ be an isomorphism of Hopf
algebras. If $e\notin\supp V$, then $\Theta\phi (V)=\phi'(V)$. \qed
\end{enumerate}
\end{lem}

\begin{proof}
The lemma follows from \cite[Lemma 4.8]{AAnMGV}
since $\phi(M^\times)\subset\phi(V\#\ku^G)$, $\phi(M[e])\subset A_1[e]$ and 
$A_1[e]=\ku^G$ 
if $e\notin\supp V$.
\end{proof}

Under certain conditions, it is showed in \cite[Section 4]{AAnMGV} that $\cI_M$ defines 
the lifting $A$. We recall this in our case.

A \emph{good module of relations} \cite[Definition 4.10]{AAnMGV} is a graded submodule
$M=\bigoplus_{i=1}^{t}M^{n_i}$ of $T(V)$ in $\ydgdual$,
$M^{n_i}\subset V^{\ot n_i}$, such that it generates $\cJ(V)$ and for all $s=1, \dots, 
t-1$ and $m\in M^{n_{s+1}}$: $n_s<n_{s+1}$, $M^{n_s}\neq 0$ and
\begin{align*}
\Delta(m)-m\ot1-m\_{-1}\ot m\_{0}\in I_{N}\ot T(V)\#\ku^G + T(V)\#\ku^G\ot I_{N}
\end{align*}
where $N=\bigoplus_{i=1}^{s} M^{n_i}$ and $M$ turn out to be compatible with $\phi$ by 
\cite[Lemma 4.9]{AAnMGV}. The next result is \cite[Theorem 4.11]{AAnMGV}. Recall \eqref{eq:IM}.

\begin{theorem}\label{cor:generadors of ker phi}
Let $A$ be a lifting of $\B(V)$ over $\ku^G$ with lifting map $\phi$. Let $M$ be a good 
module of relations for $\B(V)$. Then $A\simeq T(V)\#\ku^G/\cI_M$.\qed
\end{theorem}

\section{Pointed Hopf algebras over affine racks}\label{sec:pointed}

Let $\Aff(\F_b,N)$ be one of the affine racks $\Aff(\F_3,2)$, $\Aff(\F_4,\omega)$, $\Aff(\F_5,2)$ or 
$\Aff(\F_5,3)$. Through this section, we fix a finite group $G$ together with a principal 
YD-realization 
$(\cdot,g,\{\chi_i\}_{i\in X})$ of $(\Aff(\F_b,N),-1)$.
Let $\B(b,N,-1)$ be the Nichols algebra of $V=\ku\{x_i\}_{i\in\F_b}$ in 
$\ydg$ 
given by \eqref{eqn:yetter-drinfeld}. In this section we prove Main Theorem  1 using 
the Strategy of \cite{AAnMGV} described in \ref{general results 
groups}.

\smallbreak

Recall from Subsection \ref{subsec:nichols} a set of generators of the ideal 
$\cJ(b,N,-1)$ and set $z=z_{(b,N,-1)}$ the top degree generator. Then the hypotheses 
of 
Subsection \ref{general results groups} hold for these Nichols algebras. Namely,
\begin{itemize}
\item $\cJ(b,N,-1)$ admits an stratification as in
\eqref{eqn:stratification}. 
\item $z$ satisfies \eqref{eq:z is skewprimi} and also \eqref{eqn:znormal1} by Lemmas 
\ref{le:f max degree rel} and \ref{le:z is central}.
\item  Equations \eqref{eqn:gin} and \eqref{eqn:g(z)} hold by Lemma \ref{le:neq in group 
rack}.
\end{itemize}

Therefore we can apply Theorem \ref{thm:all-liftings} to compute the liftings of 
$\B(b,N,-1)$ over $\ku G$ once we have proved that
\begin{itemize}
\item The algebras in \eqref{eqn:galoisnonula} are nonzero.
\end{itemize}
In the next subsections, we do this using \cite{GAP, GBNP}. We stick to the notation in Subsection \ref{general results 
groups}. Recall the definition of the sets $\Ss_X$ in \eqref{eqn:Sx}, \eqref{eqn:Sxprima}.

\subsection{Pointed Hopf algebras over $\Aff(\F_3,2)$}
 Let $(\lambda_1,\lambda_2)\in\Ss_{\Aff(\F_3,2)}$. Let
$A(\lambda_1,\lambda_2)$ be the
quotient of $T(V)\#\ku G$
by the ideal generated by 
\begin{align*}
x_0^2 - \lambda_1  \quad \text{and} \quad x_0x_1+x_1x_2+x_2x_0 - \lambda_2.
\end{align*}
Let $H(\lambda_1,\lambda_2)$ be the quotient of $T(V)\#\ku G$ by the
ideal generated by
\begin{align*}
x_0^2-\lambda_1(1-g_0^2) \quad \text{ and}\quad x_0x_1+x_1x_2+x_2x_0-\lambda_2(1-g_0g_1).
\end{align*}

\begin{rem}
The pointed Hopf algebras over $\s_3$ were classified in \cite{AHS,AG3}. These are 
isomorphic either to $\s_3$ or to some $H(\lambda_1,\lambda_2)$. In  \cite{GIM} it 
was shown that the nontrivial liftings are cocycle deformations of the 
bosonization $\B(3,2,-1)\#\ku \s_3$. We give a different proof of this facts in  
Theorem \ref{teo:f3}. Also, items (a) and (d) of this theorem are already in 
\cite[Theorem 3.8]{AG3}, by different methods. 
\end{rem}

\begin{theorem}\label{teo:f3}
Let $H$ be a lifting of $\B(3,2,-1)$ over $\ku G$. 
\begin{enumerate}\renewcommand{\theenumi}{\alph{enumi}}\renewcommand{\labelenumi}{
(\theenumi)}
\item There exists $(\lambda_1,\lambda_2)\in\Ss_{\Aff(\F_3,2)}$ such that $H\cong
H(\lambda_1,\lambda_2)$.
\smallbreak
\item $A(\lambda_1,\lambda_2)$ is a $(H(\lambda_1,\lambda_2),\B(3,2,-1)\#\ku 
G)$-biGalois
object for every $(\lambda_1,\lambda_2)\in\Ss_{\Aff(\F_3,2)}$.
\item $H$ is a cocycle deformation of $\B(3,2,-1)\#\ku G$.
\item $H(\lambda_1,\lambda_2)$ is a lifting of $\B(3,2,-1)$ over $\ku G$ for every
$(\lambda_1,\lambda_2)\in\Ss_{\Aff(\F_3,2)}$.
\item $H(\lambda_1,\lambda_2)\cong H(\lambda_1',\lambda_2')$ iff
$(\lambda_1,\lambda_2)=\mu(\lambda_1',\lambda_2')$ for some $\mu\in\ku$.
\end{enumerate}
\end{theorem}

\pf
Follows by Corollary \ref{cor:all-liftings}. We consider the stratification of
$\mJ(3,2,-1)$ given by 
$\mG_0=\{x_i^2:\,i\in\F_3\}$ and $\mG_1=\{x_ix_j+x_{-i+2j}x_i+x_jx_{-i+2j}:\,i,j\in\F_3\}$
and then we use Diamond Lemma to see that $A(\lambda_1,\lambda_2)\neq 0$.
\epf

\subsection{Pointed Hopf algebras over $\Aff(\F_4,\omega)$}\label{subsec:liftingsf4}

Let $(\lambda_1,\lambda_2,\lambda_3)\in\Ss_{\Aff(\F_4,\omega)}$.
Let
$A(\lambda_1,\lambda_2,\lambda_3)$ be the quotient of
$T(V)\#\ku G$ by the ideal generated by
\begin{align*}
&x_0^2 - \lambda_1,\qquad  x_0x_1 + x_1x_2 + x_2x_0 - \lambda_2 \quad \text{and}\\
& (x_0x_1x_2)^2 + (x_1x_2x_0)^2+(x_2x_0x_1)^2 -a_X - \lambda_3 \quad \text{for}
\end{align*}
$a_X=\lambda_2(x_1x_0x_2x_1+ x_0x_2x_1x_0+x_2x_1x_0x_2) +\lambda_2(\lambda_2
-\lambda_1)(x_2x_1  +  x_1x_0 + x_0x_2).$

Let $H(\lambda_1,\lambda_2,\lambda_3)$ be the quotient of $T(V)\#\ku G$
by the
ideal generated by
\begin{align}
\notag & x_0^2 - \lambda_1(1- g_0^2), \qquad   x_0x_1 + x_1x_2 + x_2x_0 - \lambda_2(1 -
g_0g_1) \quad  \text{and}\\
\notag &  x_2x_1x_0x_2x_1x_0 + x_1x_0x_2x_1x_0x_2 +x_0x_2x_1x_0x_2x_1
-s_X - \lambda_3(1-g_0^3g_1^3)
\end{align}
where
\begin{align*} &s_X=\lambda_2(x_2x_1x_0x_2+ x_1x_0x_2x_1
+x_0x_2x_1x_0)-\lambda_2^3(g_0g_1-g_0^3g_1^3)\\
 & \, +\lambda_1^2g_0^2\big(g_3^2(x_2x_3 +x_0x_2) +g_1g_3(x_2x_1+
x_1x_3)
+g_1^2(x_1x_0+x_0x_3)\big)
\\
  & \, - 2\lambda_1^2g_0^2(x_0x_3- x_2x_3   -x_1x_2 + x_1x_0) - 2\lambda_1^2g_2^2(x_2x_3
-x_1x_3+x_0x_2 - x_0x_1 )\\
 &\, - 2\lambda_1^2g_1^2(x_2x_1 +  x_1x_3 +  x_1x_2 - x_0x_3 + x_0x_1) \\
  & \, +\lambda_2\lambda_1(g_2^2x_0x_3+ g_1^2x_2x_3+ g_0^2x_1x_3)
+\lambda_2^2g_0g_1(x_2x_1 + x_1x_0 + x_0x_2-\lambda_1)\\
   &\, -
\lambda_2\lambda_1^2(3g_0^3g_3-2 g_0g_1^3- g_0^2g_2^ 2
-2 g_0^3g_1+g_2^2 - g_1^2+g_0^2)\\
& \, 
 -\lambda_2(\lambda_1-\lambda_2)\big(\lambda_1\,g_0^2(g_3^2+g_1g_3 +g_1^2+2g_0g_1^3)
+x_2x_1 + x_1x_0 + x_0x_2\big).
\end{align*}

\begin{theorem}\label{teo:f4}
Let $H$ be a lifting of $\B(4,\omega,-1)$ over $\ku G$. 
\begin{enumerate}\renewcommand{\theenumi}{\alph{enumi}}\renewcommand{\labelenumi}{
(\theenumi)}
\item There exists $(\lambda_1,\lambda_2,\lambda_3)\in\Ss_{\Aff(\F_4,\omega)}$ such that
$H\cong
H(\lambda_1,\lambda_2,\lambda_3)$.
\item $A(\lambda_1,\lambda_2,\lambda_3)$ is a
$(H(\lambda_1,\lambda_2,\lambda_3),\B(4,\omega,-1)\#\ku G)$-biGalois object for every
$(\lambda_1,\lambda_2,\lambda_3)\in\Ss_{\Aff(\F_4,\omega)}$.
\item $H$ is a cocycle deformation of $\B(4,\omega,-1)\#\ku G$.
\item $H(\lambda_1,\lambda_2,\lambda_3)$ is a lifting of $\B(4,\omega,-1)\#\ku G$ for
all $(\lambda_1,\lambda_2,\lambda_3)\in\Ss_{\Aff(\F_4,\omega)}$.
\item $H(\lambda_1,\lambda_2,\lambda_3)\cong
H(\lambda_1',\lambda_2',\lambda_3')$ iff
$(\lambda_1,\lambda_2,\lambda_3)=\mu(\lambda_1',\lambda_2',\lambda_3')$ for some $\mu\in\ku$.
\end{enumerate}
\end{theorem}
\pf
The algebras $H(\lambda_1,\lambda_2,\lambda_3)$ are found following the strategy described
in Subsection \ref{general
results groups}. We check that the algebras $\mA_2(\lambda_1,\lambda_2)$ are nonzero
using \cite{GAP,GBNP}. We compute $\gamma_2(z)$, for $\gamma_2:\mH_2\to 
\mA_2(\lambda_1,\lambda_2)$ as in
Proposition \ref{pro:lifting1} (b), again
using \cite{GAP,GBNP}, as
explained in the Appendix. We
end up with the liftings $H(\lambda_1,\lambda_2,\lambda_3)$ using Proposition
\ref{pro:liftings3}, which states
(b) and (d), consequently (c) and (e). Now (a) follows from Theorem 
\ref{thm:all-liftings}.  
\epf
\subsection{Pointed Hopf algebras over $\Aff(\F_5,2)$}Let 
$(\lambda_1,\lambda_2,\lambda_3)\in\Ss_{\Aff(\F_5,2)}$. 
Let
$A(\lambda_1,\lambda_2,\lambda_3)$ be the quotient of
$T(V)\#\ku G$ by the ideal generated by
\begin{align*}
&x_0^2  -\lambda_1, \quad x_0x_1 + x_2x_0 + x_3x_2 + x_1x_3   - \lambda_2, \\
(x_0x_1)^2 & + (x_1x_0)^2 - \lambda_2\, (x_1x_0 + x_0x_1)- \lambda_3.
\end{align*}
Let
$H(\lambda_1,\lambda_2,\lambda_3)$ be the quotient of $T(V)\#\ku G$
by the ideal generated by
\begin{align*}
&x_0^2 - \lambda_1 (1- g_0^2),  \qquad x_0x_1 + x_2x_0 + x_3x_2 + x_1x_3 - \lambda_2(1 -
g_0g_1) \text{ and } \\
&x_1x_0x_1x_0 + x_0x_1x_0x_1 -s_X - \lambda_3(1-g_0^2g_1g_2),
\end{align*}
for $s_X=\lambda_2\, (x_1x_0 +
x_0x_1)+\lambda_1\, g_1^2(x_3x_0+ x_2x_3) - \lambda_1\, g_0^2(x_2x_4+ x_1x_2) +
\lambda_2\lambda_1\,g_0^2(1- g_1g_2)$. 
\begin{theorem}\label{teo:f5}
Let $H$ be a lifting of $\B(5,2,-1)$ over $\ku G$. 
\begin{enumerate}\renewcommand{\theenumi}{\alph{enumi}}\renewcommand{\labelenumi}{
(\theenumi)}
\item There exists $(\lambda_1,\lambda_2,\lambda_3)\in\Ss_{\Aff(\F_5,2)}$ such that $H\cong
H(\lambda_1,\lambda_2,\lambda_3)$.
\item $A(\lambda_1,\lambda_2,\lambda_3)$ is a
$(H(\lambda_1,\lambda_2,\lambda_3),\B(5,2,-1)\#\ku G)$-biGalois object for every
$(\lambda_1,\lambda_2,\lambda_3)\in\Ss_{\Aff(\F_5,2)}$.
\item $H$ is a cocycle deformation of $\B(5,2,-1)\#\ku G$.
\item $H(\lambda_1,\lambda_2,\lambda_3)$ is a lifting of $\B(5,2,-1)\#\ku G$ for
every
$(\lambda_1,\lambda_2,\lambda_3)\in\Ss_{\Aff(\F_5,2)}$.
\item $H(\lambda_1,\lambda_2,\lambda_3)\cong
H(\lambda_1',\lambda_2',\lambda_3')$ iff
$(\lambda_1,\lambda_2,\lambda_3)=\mu(\lambda_1',\lambda_2',\lambda_3')$ for some $\mu\in\ku$.
\end{enumerate}
\end{theorem}
\pf
Set $z=(x_0x_1)^2 + (x_1x_0)^2 \in\mA'=\mA_2(\lambda_1,\lambda_2)$, $t_z=g_0^2g_1g_2\in
G$. 
Using \cite{GAP,GBNP}\footnote{See log files in 
\texttt{http://www.mate.uncor.edu/$\sim$aigarcia/publicaciones.htm}.}, the coaction of 
$z$ in $\mA'$ is $\delta_R^2(z)=z\ot1+
t_z\ot z$ plus:
\begin{align*}
&  \lambda_2\, g_0g_3\ot x_1x_0 + \lambda_2\, g_0g_1\ot x_0x_1  -\lambda_2\, g_1x_3\ot
x_1 \\
&\quad + \lambda_2\, g_1x_0\ot x_1  -\lambda_2\, g_0x_3\ot x_0 + \lambda_2\, g_0x_1\ot
x_0.
\end{align*}
If $z'= x_1x_0 + x_0x_1$ we get $\delta_R^2(z -\lambda_2\, z')=(z -\lambda_2\, z')\ot 1+t_z\ot
z$. Thus $\gamma_2(z)=z-\lambda_2\, z'$ and the theorem follows as Theorem \ref{teo:f4}.
\epf

\subsection{Pointed Hopf algebras over $\Aff(\F_5,3)$}

Let $(\lambda_1,\lambda_2,\lambda_3)\in\Ss_{\Aff(\F_5,3)}$.
Let
$A(\lambda_1,\lambda_2,\lambda_3)$ be the quotient of
$T(V)\#\ku G$ by the ideal generated by
\begin{align*}
&x_0^2  -\lambda_1, \quad x_1x_0 + x_0x_2 + x_2x_3 + x_3x_1  - \lambda_2, \\
(x_0x_1)^2 &+ (x_1x_0)^2 - \lambda_2\, (x_0x_1 + x_1x_0) - \lambda_3.
\end{align*}
Let
$H(\lambda_1,\lambda_2,\lambda_3)$ be the quotient of $T(V)\#\ku G$
by the ideal generated by
\begin{align*}
&x_0^2 - \lambda_1 (1- g_0^2),  \qquad x_1x_0 + x_0x_2 + x_2x_3 + x_3x_1 - \lambda_2(1 -
g_0g_1) \text{ and } \\
&x_0x_2x_3x_1 + x_1x_4x_3x_0 -s_X - \lambda_3(1-g_0^2g_1g_3),
\end{align*}
for $s_X=\lambda_2\, (x_0x_1 + x_1x_0) -\lambda_1\, g_1^2(x_3x_2 +x_0x_3) - \lambda_1\, 
g_0^2(x_3x_4 +x_1x_3)  + \lambda_1\lambda_2( g_1^2 + 
g_0^2- 2 g_0^2g_1g_3)$. 
\begin{theorem}\label{teo:f5-3}
Let $H$ be a lifting of $\B(5,3,-1)$ over $\ku G$. 
\begin{enumerate}\renewcommand{\theenumi}{\alph{enumi}}\renewcommand{\labelenumi}{
(\theenumi)}
\item There exists $(\lambda_1,\lambda_2,\lambda_3)\in\Ss_{\Aff(\F_5,3)}$ such that $H\cong
H(\lambda_1,\lambda_2,\lambda_3)$.
\item $A(\lambda_1,\lambda_2,\lambda_3)$ is a
$(H(\lambda_1,\lambda_2,\lambda_3),\B(5,3,-1)\#\ku G)$-biGalois object for every
$(\lambda_1,\lambda_2,\lambda_3)\in\Ss_{\Aff(\F_5,3)}$.
\item $H$ is a cocycle deformation of $\B(5,3-1)\#\ku G$.
\item $H(\lambda_1,\lambda_2,\lambda_3)$ is a lifting of $\B(5,3,-1)\#\ku G$ for
every 
$(\lambda_1,\lambda_2,\lambda_3)\in\Ss_{\Aff(\F_5,3)}$.
\item $H(\lambda_1,\lambda_2,\lambda_3)\cong
H(\lambda_1',\lambda_2',\lambda_3')$ iff
$(\lambda_1,\lambda_2,\lambda_3)=\mu(\lambda_1',\lambda_2',\lambda_3')$ for some $\mu\in\ku$.
\end{enumerate}
\end{theorem}
\pf
Analogous to Theorem \ref{teo:f5-3} {\it mutatis mutandis}.
\epf

\subsection{Proof of Main Theorem 1}
Assume $X=\Aff(\F_3,2)$. Let $H$ be a pointed Hopf algebra over $G$ whose infinitesimal 
braiding is given by a principal YD-realization $V\in\ydg$ of $(X,-1)$. Then $H$ is generated in 
degree one by \cite[Theorem 2.1]{AG3}. Therefore $H$ is a lifting of $\B(V)$ over $\ku G$ 
and Main Theorem 1 (i) follows by Theorem \ref{teo:f3}.

The proof of items (ii), (iii), (iv) is analogous, again using the fact that 
any such $H$ is generated in degree one by \cite[Theorem 2.1]{AG3} and Theorems 
\ref{teo:f4}, \ref{teo:f5} or \ref{teo:f5-3}, depending on each case.
\qed

\section{Copointed Hopf algebras over affine racks}

Through this section, we consider the affine racks $\Aff(\F_b,N)$ with constant $2$-cocycle $-1$.  
We fix a finite group $G$ and a principal YD-realization 
$(\cdot,g,\{\chi_i\}_{i\in X})$ of $(\Aff(\F_b,N),-1)$ over $G$. Let $\B(-1,b,N)$ be the 
Nichols algebra of $W(-1,b,N)=\ku\{x_i\}_{i\in\F_b}$ in $\ydgdual$ given by 
\eqref{eqn:yetter-drinfeld-dual}.
We give the classification of the lifting Hopf algebras of $\B(-1,b,N)$ over
$\ku^G$ and therefore the proof of Main Theorem  2.

\subsection{Copointed Hopf algebras over $\Aff(\F_3,2)$} 
This subsection is inspired by \cite{AV,AV2} where the case $G=\Sn_3$ was considered.
Recall that $\Inn_\rhd\Aff(\F_3,2)=\Sn_3=\Aut_{\rhd}\Aff(\F_3,2)$ by \eqref{eq:inn equal aut}.
Let 
$G\longrightarrow\Sn_3$, $t\mapsto\overline{t}$ be the epimorphism given by Lemma 
\ref{le:the enveloping se proyecta} \eqref{ite:Autle:the enveloping se proyecta}.
We consider the group $\Gamma=\ku^*\times\Sn_3$ acting on
$$\gA=\bigl\{\ba=(a_0,a_1,a_2)\in\ku^{\F_3}:a_0+a_1+a_2=0\bigr\}$$
via $(\mu,\theta)\triangleright\ba=\mu(a_{\theta0},a_{\theta1},a_{\theta2})$. The
equivalence class of $\ba$ under this
action is denoted by $[\ba]$. Given
$\ba\in\gA$, we define
\begin{align*}
f_{i}=\sum_{t\in G}(a_i - a_{t^{-1}\cdot i})\,\delta_t \in
\ku^G,\quad i\in\F_3.
\end{align*}

\begin{fed}\label{def: Affin 3} Set $\cA_{G,[0]}=\B(-1,3,2)\#\ku^G$. Let
$\ba\in\gA$ and assume that $g_i^2=e$ $\forall i\in\F_3$. We
define the Hopf algebra 
$\cA_{G,[\ba]}=T(-1,3,2)\#\ku^G/\cJ_{\ba}$ where $\cJ_{\ba}$ is the ideal generated
by 
\begin{align*}
x_i^2-f_i, \quad x_ix_j+x_{-i+2j}x_i+x_jx_{-i+2j},\quad i,j\in\F_3,
\end{align*}
and the algebra $\cK_{G,[\ba]}=T(-1,3,2)\#\ku^G/\cI_{\ba}$ where $\cI_{\ba}$ is 
generated
by 
\begin{align*}
x_i^2+f_i-a_i, \quad x_ix_j+x_{-i+2j}x_i+x_jx_{-i+2j},\quad i,j\in\F_3.
\end{align*}
\end{fed}

The algebras $\cA_{G,[\ba]}$ and $\cK_{G,[\ba]}$ are nonzero by the next lemma.

\begin{lem}\label{le:rep of Lga y Aga}
Consider the $\ku^G$-module $M=\ku\{m_t\}_{t \in G}$, $m_t\in M[t]$. Then 
for all
$\ba\in\gA$, $M$
is an $\cA_{G,[\ba]}$-module and a $\cK_{G,[\ba]}$-module via
\begin{align*}
\quad x_i\cdot m_t=\begin{cases}
  m_{g_i^{-1}t} & \mbox{ if }\sgn\overline{t}=-1,\\
\lambda_{i,t}\,m_{g_i^{-1}t} & \mbox{ if }\sgn\overline{t}=1.\\
\end{cases}
\end{align*} 
where $\lambda_{i,t}=(a_i-a_{t^{-1}\cdot i})$ for $\cA_{G,[\ba]}$ and 
$\lambda_{i,t}=-a_{t^{-1}\cdot i}$ for $\cK_{G,[\ba]}$.
\end{lem}
\begin{proof}
We check that the action of $\cK_{G,[\ba]}$ is well-defined; for
$\cA_{G,[\ba]}$ it is similar. Notice that $\sgn(\overline{g_i})=-1$. We start by
$\delta_h
x_i=x_i\delta_{g_ih}$, cf. \eqref{eq:TV smash ku a la G}:
\begin{align*}
\delta_h(x_i\cdot m_t)&=\delta_h(\lambda\, m_{g_i^{-1}t})=\lambda\delta_{g_ih}(t) 
m_{g_i^{-1}t}=x_i\cdot(\delta_{g_ih}\cdot m_{t})
\end{align*}
for a certain $\lambda\in\ku$. Clearly $x_i\cdot(x_i\cdot m_t)=\lambda_{i,t}\, m_t$.
Since $\ba\in\gA$, then
$
(x_ix_j+x_{-i+2j}x_i+x_jx_{-i+2j})\cdot m_t=-(a_0+a_1+a_2)\,m_{g_i^{-1}g_j^{-1}t}=0.
$ 
\end{proof}

The following theorem presents all the liftings of 
$\B(-1,3,2)$ over $\ku^G$.

\begin{theorem}\label{prop: clas of F3 2}
Let $H$ be a lifting of $\B(-1,3,2)$ over $\ku^G$. 
\begin{enumerate}\renewcommand{\theenumi}{\alph{enumi}}\renewcommand{\labelenumi}{
(\theenumi)}
\item\label{item:no deforma g gen F3 2} If $g_i^2\neq e$ for some (and thus all)
$i\in\F_3$, then $H\simeq\cA_{G,[0]}$.
\smallbreak
\item\label{item:si deforma g gen F3 2} If $g_i^2= e$ for some (and thus all) $i\in\F_3$,
then there is $\ba\in\gA$ such that $H\simeq\cA_{G,[\ba]}$.
\smallbreak
\item\label{item:Agba is galois bi-galois F3 2} $\cK_{G,[\ba]}$ is a
$(\cA_{G,[0]},\cA_{G,[\ba]})$-biGalois object for all $\ba\in\gA$.
\smallbreak
\item\label{item:Lgba is cocycle deformation F3 2} $\cA_{G,[\ba]}$ is a cocycle
deformation of $\cA_{G,[\bb]}$ for all $\ba,\bb\in\gA$.
\smallbreak
\item\label{item:Lgba is lifting F3 2} $\cA_{G,[\ba]}$ is a lifting of
$\B(-1,3,2)$ over $\ku^G$ for all $\ba,\bb\in\gA$.
\smallbreak
\item\label{item:Lgba iso class F3 2} $\cA_{G,[\ba]}\simeq\cA_{G,[\bb]}$ if and only if
$[\ba]=[\bb]$.
\end{enumerate}
\end{theorem}

\begin{proof}
Let $\phi:T(-1,3,2)\#\ku^G\rightarrow H$ be a lifting map and let 
$$
W=\{ x_i^2,\, x_ix_j+x_{-i+2j}x_i+x_jx_{-i+2j}: i,j\in\F_3\}
$$
be the set of quadratic relations defining the Nichols
algebra $\B(-1,3,2)$, see Proposition \ref{pro:rels-dual}. Let $M\subset 
T(-1,3,2)$ 
be 
the
Yetter-Drinfeld submodule generated by $W$. Then
$\phi(M[g_i^{-1}g_j^{-1}])=0$ by Lemma
\ref{thm:liftings} \eqref{cor:liftings Mtimes} using Lemma \ref{le:neq in group rack}
\eqref{item:general le:neq in group rack} and (c). Hence: 

\eqref{item:no deforma g gen F3 2} follows from Lemma \ref{thm:liftings}
\eqref{cor:liftings
Mtimes} and Theorem \ref{cor:generadors of ker phi} using Lemma \ref{le:neq in group rack}
\eqref{item:general le:neq in group rack}.

\eqref{item:si deforma g gen F3 2} follows from Lemma \ref{thm:liftings}
\eqref{cor:liftings Me}
and Theorem \ref{cor:generadors of ker phi}.

\eqref{item:Agba is galois bi-galois F3 2} follows from \cite[Theorem 2]{masuoka}. In
fact, fix $\ba\in\gA$ and let $K$ be the braided Hopf subalgebra
of $T(-1,3,2)$ generated by $W$. Then $K\#\ku^{G}$ is a Hopf subalgebra of 
$T(-1,3,2)\#\ku^{G}$. By \cite[Lemma 28]{AV2}, we can define an algebra map
$\psi=\psi_K\otimes\epsilon:K\#\ku^{G}\rightarrow\ku$ where
$$
\psi_K(x_i^2)=-a_i\,\mbox{ and }\,\psi_{K}(x_ix_j+x_{-i+2j}x_i+x_jx_{-i+2j})=0\quad\forall
i,j\in\F_3.
$$
If $J=\langle W\rangle \subset K\#\ku^{G}$, then $\psi^{-1}\rightharpoonup
J\leftharpoonup\psi=\cJ_{\ba}$ and $\psi^{-1}\rightharpoonup J=\cI_{\ba}$. By Lemma
\ref{le:rep of Lga y Aga}, $\cK_{G,[\ba]}\neq0$ and  \cite[Theorem 2]{masuoka}
asserts \eqref{item:Agba is galois bi-galois F3 2}; hence \eqref{item:Lgba is cocycle
deformation F3 2} and \eqref{item:Lgba is lifting F3 2}.

\eqref{item:Lgba iso class F3 2} Fix $\ba,\bb\in\gA$. Let
$\phi_{\ba}$ and $\phi_{\bb}$ be lifting maps of $\cA_{G,[\ba]}$ and $\cA_{G,[\bb]}$.
Let
$\Theta:\cA_{G,[\ba]}\rightarrow\cA_{G,[\bb]}$ be an isomorphism of Hopf algebras. Then
$(\Theta_{|\ku^{G}})^*$ induces a group automorphism $\theta$ of $G$. By Lemma
\ref{thm:liftings} \eqref{item:iso between As} and
using the adjoint action of $\ku^{G}$, we see that $\theta$ is a rack automorphism of
$\Aff(\F_3,2)$ and
$\Theta\phi_{\ba}(x_i)=\mu_i\phi_{\bb}(x_{\theta i})$ with
$\mu_{i}\in\ku^*$ for all $i\in\F_3$. Since $\Theta$ is a coalgebra map, using
\eqref{eq:TV smash ku a la G}
we obtain that $\mu_{i}=\mu$ for all $i\in\F_3$. Therefore
$\ba=(\mu^2,\theta)\triangleright\bb$. The proof of the converse statement is easy, recall
that $\Inn_\rhd\Aff(\F_3,2)=\Sn_3$.
\end{proof}

\subsection{Copointed Hopf algebras over $\Aff(\F_b,N)$} 
Here $\Aff(\F_b,N)$ denotes one the racks $\Aff(\F_4,\omega)$, $\Aff(\F_5,2)$, $\Aff(\F_5,3)$, $\Aff(\F_7,3)$ 
or $\Aff(\F_7,5)$. Recall that $\chi_G=\chi_i$ is a multiplicative character for all $i\in X$ 
by Lemma 
\ref{le:the enveloping se proyecta} \eqref{ite:chiG:the enveloping se proyecta}.   Let 
$\pi_2:T(-1,b,N)\twoheadrightarrow\widehat{\B_2}(-1,b,N)$ be the natural 
projection. 
Set $z'=z'_{(-1,b,N)}$ and $\chi_z=\chi_G^{\deg z}$, recall \eqref{eqn:chiz dual}.

\begin{fed}\label{def:Lglambda}
Set $\cA_{G,0}=\B(-1,b,N)\#\ku^G$. If $z'\in T(-1,b,N)[e]$ and 
$\lambda\in\ku^*$, 
then we define the Hopf algebra 
$$\cA_{G,\lambda}=\widehat{\B_2}(-1,b,N)\#\ku^G/\langle\pi_2(z')-\lambda(1-\chi_z^{-1}
)\rangle$$
and the algebra
$\cK_{G,\lambda}=\widehat{\B_2}(-1,b,N)\#\ku^G/\langle\pi_2(z')-\lambda\rangle$.
\end{fed}

The following theorem presents all the liftings of $\B(-1,b,N)$ over $\ku^G$.

\begin{theorem}\label{thm:lifting over affines grandes}
Let $H$ be a lifting of $\B(-1,b,N)$ over $\ku^G$.
\begin{enumerate}\renewcommand{\theenumi}{\alph{enumi}}\renewcommand{\labelenumi}{
(\theenumi)}
\item\label{item:no deforma g gen thm:lifting over affines grandes} If $G$ is generated by
$\{g_i^{-1}:i\in\F_b\}$ or $\chi_z=\varepsilon$, then $H\simeq\cA_{G,0}$.
\smallbreak
\item\label{item:no deforma no e thm:lifting over affines grandes} If $z'\in
T(-1,b,N)^\times$, $H\simeq\cA_{G,0}$.
\smallbreak
\item\label{item:quadratic thm:lifting over affines grandes} If $z'\in T(-1,b,N)[e]$, 
then
$H\simeq\cA_{G,\lambda}$ for some $\lambda\in\ku$.
\smallbreak
\item\label{item:Aglambda is galois bi-galois thm:lifting over affines grandes}
$\cK_{G,\lambda}$ is a $(\cA_{G,0},\cA_{G,\lambda})$-biGalois object
for all $\lambda\in\ku$.
\smallbreak
\item\label{item:Lglambda is cocycle deformation thm:lifting over affines grandes}
$\cA_{G,\lambda}$ is a cocycle deformation of $\cA_{G,\lambda'}$,
for all $\lambda, \lambda'\in\ku$.
\smallbreak
\item\label{item:Lglambda lifting thm:lifting over affines grandes}
$\cA_{G,\lambda}$ is a lifting of $\B(-1,b,N)$ over $\ku^G$
for all $\lambda, \lambda'\in\ku$.
\smallbreak
\item\label{item:Lglambda iso class thm:lifting over affines grandes}
$\cA_{G,\lambda}\simeq\cA_{G,1}\not\simeq\cA_{G,0}$ for all $\lambda\in\ku$.
\end{enumerate}
 \end{theorem}

\begin{proof}
Let $\phi:T(-1,b,N)\#\ku^G\rightarrow H$ be a lifting map and $M\subset T(-1,b,N)$ 
be 
the
Yetter-Drinfeld submodule generated by the quadratic relations defining $\B(-1,b,N)$,
see Proposition \ref{pro:rels-dual}. Then 
$$M=M^\times=\bigoplus_{i,j\in\F_b}M[(g_ig_j)^{-1}]\oplus\bigoplus_{i\in\F_b}M[g_i^{-2}]
$$ 
by Lemma \ref{le:neq in group rack}. Moreover, $\phi(M^\times)=0$ by Lemma
\ref{thm:liftings} \eqref{cor:liftings Mtimes} using Lemma \ref{le:neq in group rack}.
Therefore $\phi$ factorizes through $\widehat{\B_2}(-1,b,N)\#\ku^G$ and the
Yetter-Drinfeld module $M_{z'}$ generated by $z'$ is compatible with $\phi$. Therefore:

\eqref{item:quadratic thm:lifting over affines grandes} follows from Lemma
\ref{thm:liftings} \eqref{cor:liftings Me} and Theorem \ref{cor:generadors of ker phi} by \eqref{eq:z is skewprimi dual} . \eqref{item:no deforma g gen thm:lifting over affines grandes}
follows from
\eqref{item:quadratic thm:lifting over affines grandes} since $\chi_z=\chi_G^{\deg
z}=\varepsilon$ by Lemma \ref{le:f max degree rel}.
\eqref{item:no deforma no e thm:lifting over affines grandes} follows from Lemma
\ref{thm:liftings} \eqref{cor:liftings Mtimes} and Theorem \ref{cor:generadors of ker phi}
since $1=\dim\pi_2(M_{z'})<\dim W(-1,b,N)$; the equality holds by Lemma \eqref{eqn:chiz dual}.  

\eqref{item:Aglambda is galois bi-galois thm:lifting over
affines grandes} 
Let $w=\pi_2(z')\chi_z$. By Lemma \ref{le:f max degree rel prima} $\ku[w]$ is the 
subalgebra of right $\B(-1,b,N)\#\ku^G$-coinvariants.
By \cite[Corollaries 3.7 and 3.8]{AAnMGV} we can apply \cite[Theorem 4]{Gu} to the 
Yetter-Drinfeld algebra map
$$\ku[w]\longrightarrow \widehat{\B_2}(-1,b,N)\#\ku^G,\ w\mapsto w-\lambda\chi_z.$$
Hence $\cK_{G,\lambda}$ is a
$(\cA_{G,\lambda},\cA_{G,0})$-biGalois object. \eqref{item:Lglambda is cocycle
deformation thm:lifting
over affines grandes} and \eqref{item:Lglambda lifting thm:lifting over affines grandes}
are consequences of \eqref{item:Aglambda is galois bi-galois thm:lifting over affines
grandes}. For \eqref{item:Lglambda iso class thm:lifting over affines grandes}, the map
$F:\cA_{G,\lambda}\longrightarrow\cA_{G,1}$ given by $F(x_i)=\lambda^{1/\deg z}x_i$ and
$F_{|\ku^G}=\id_{|\ku^G}$ is an isomorphism of Hopf algebras.
\end{proof}

\begin{exa}
There are nontrivial liftings of $\B(-1,b,N)$ isomorphic to
$\cA_{G,\lambda}$. In fact, suppose that $m\mid\ell k+1$ and consider the $(m,k)$-affine
realization of $(\Aff(\F_b,N),-1)$; note that $z'\in T(-1,b,N)[e]$. Let $G'$ be a
finite group with a multiplicative character
$\chi_{G'}:G'\rightarrow\ku^*$ such that $\chi_{G'}^{\deg
z}\neq\varepsilon$. Then 
$G=(\F_b\rtimes C_{m\ell})\times G'$ acts on $W(-1,b,N)$ via
$(h\times g')\cdot x_i=\chi_{G'}(g')h\cdot x_i$ and thus  the
$(m,k)$-affine realization induces a principal YD-realization of $(\Aff(\F_b,N),-1)$ over $G$
such that $z'\in T(-1,b,N)[e]$ and $\chi_z=\chi_G^{\deg z}\neq\varepsilon$.
\end{exa}

\subsection{Proof of Main Theorem 2}

Let $H$ be a copointed Hopf algebra over $\ku^G$ whose infinitesimal braiding is given by 
a principal YD-realization $W(-1,b,N)\in\ydgdual$. Then $H$ 
is generated in degree one. Indeed, we can repeat the proof of \cite[Theorem 2.1]{AG3}, 
{\it mutatis mutandis}, using the results of Subsection \ref{subsec:nichols afin sobre 
el dual}. Hence $H$ is a lifting of $\B(-1,b,N)$ over $\ku^G$ and Main Theorem 
2 follows by Theorem \ref{prop: clas of F3 2} for $\Aff(\F_3,2)$ or else by Theorem 
\ref{thm:lifting over affines grandes}. \qed

\section*{Appendix: on computations}\label{appendix:coproduct}
Through this section we keep the hypotheses and notation in Subsection \ref{general 
results groups}. We explain how we can compute using \cite{GAP,GBNP} the left and right 
coactions of the top degree relation of a Nichols algebra as there.

Set $\mH'=\mH_2$, $\mA'=\mA_2(\lambda,\mu)$ and $\mL'=\mL(\lambda,\mu)$; these are 
quotients of $T(V)\#\ku G$. Let $\delta_R$ and $\delta_L$ be the coactions on $\mA'$ over 
$\mH'$ and $\mL'$, respectively; these are induced by the comultiplication of $T(V)\#\ku 
G$.

Let $z$ be the top degree generator of $\mJ(V)$, set $\ell=\deg z$. It is an element of 
$T(V)$ but we still denote by $z$ its class in $\mH'$, $\mA'$ or $\mL'$. We compute 
\begin{itemize}
\item[(i)] The coaction $\delta_R(z)\in \mA'\ot \mH'$.
\item[(ii)] The section $\gamma_2(z)\in\mA'$.
\item[(iii)] The coaction $\delta_L(\gamma_2(z))\in \mL'\ot \mA'$.
\end{itemize}

For item (i) we proceed as follows. Let $\theta=\dim V$ and denote by $y_i$, $i=1, 
\dots, \theta$ the generators of $\mA'$. We work with $3\theta+2$ 
variables
$G_1,\dots,G_\theta$, $X_1,\dots, X_\theta$, $Y_1,\dots, Y_\theta$, 
$U,V$. The variables $G_i$ stand
for the elements $g_i\ot 1$ in
$\mA'\ot \mH'$. The variables $Y_i$ and $X_i$ stand for 
$y_i\ot 1$ and $1\ot x_i$, respectively. The variables $U$ and $V$ are included to have 
an 
homogeneous system of generators. In most cases
one of them is enough (for instance if $n=2$) and there are cases where we can 
omit them. We fix two indeterminate elements $s,t\in \ku$, corresponding to $\lambda$, 
$\mu$.

We define an ideal $K$ of relations in the algebra generated by
these variables whose generators are, for every $1\leq i,j\leq \theta$ and for every class
$C\in\mR'$:
\begin{align*}
& Y_iX_j-X_{j}Y_i, && X_iG_j-G_{j}X_i, && G_iG_j-G_{i\rhd
j}G_i &&  G_iY_j+ Y_{i\rhd j}G_i, \\
& X_i^n,&& Y_i^n-t\,U,  &&b_C(\{X_i\}_{i\in X}),  && b_C(\{Y_i\}_{i\in X})+s\, V. 
\end{align*}
Also, $U$ and $V$ commute with all $X,Y,G$. Recall that $b_C(\cdot)$ stands 
for a generator of the space of quadratic relations, see \eqref{eqn:bc}. This ideal is
homogeneous if we declare all $X, Y, G$
of degree 1, $U$ of degree $n$ and $V$ of degree 2. We compute the (truncated, up to
degree $\ell$) Gr\"obner basis of $K$.

We define $d_i=Y_i+G_iX_i$, $i=1,\dots,\theta$. 
These elements stand for the coaction of $y_i\in \mA'$. We can now
compute the coaction $\delta_R(z)$ by adding and multiplying
the $d_i$'s in a suitable way. For $U,V$
we consider it as 1.

We now explain how we get $\gamma_2(z)$ in item (ii).
Let $\delta_R(z)-t_z\ot z=\sum A_i\ot
B_i$ and let $z'$ be the sum of the terms $B_i$ of greatest degree with $A_i\in\ku G$.
Re-write the $B_i$'s in the variables $Y_i$ and consider $z_1=z-z'$. We calculate
$\delta_R(z_1)$ and repeat the proceeding: in the examples considered, the order of the
elements we subtract decreases. When
$\delta_R(z_m)-t_z\ot z\in \mA'\ot \ku G$, for some $m$, we get that
$\delta_R(z_m)=z_m\ot 1+t_z\ot z$ and thus $\gamma_2(z)=z_m$.  

Finally, we find $\delta_L(\gamma_2(z))$ in (iii) in a similar way as we did for 
$\delta_R(z)$. 

\begin{exa}\label{exa:coaction}
Let $X=(\F_4, \omega)$. We use \cite{GAP,GBNP} to see that the algebras
$\mA'=\mA_2(\lambda_1,\lambda_2)$ are nonzero. See the log files in \newline
\texttt{http://www.mate.uncor.edu/$\sim$aigarcia/publicaciones.htm}. \newline
Set $z=(y_0y_1y_2)^2 + (y_1y_2y_0)^2+(y_2y_0y_1)^2 \in\mA'$, $t_z=g_0^3g_1^3\in
G$. Using \cite{GAP,GBNP}, the coaction of $z$ in $\mA'$ is $\delta_R(z)=z\ot
1+t_z \ot z$ plus:
\begin{align*}
& \lambda_2g_0^2g_3^2\ot x_1x_0x_2x_1 + \lambda_2g_0^2g_1g_3\ot x_0x_2x_1x_0 +
\lambda_2g_0^2g_1^2\ot x_2x_1x_0x_2\\
&\,  + \lambda_2g_0g_1g_3(y_2-y_0)\ot x_2x_1x_0+
\lambda_2g_0g_1g_3(y_1-y_2)\ot x_1x_0x_2  \\
&\, +\lambda_2g_0g_1g_3(y_0 -y_1)\ot x_0x_2x_1  + \lambda_2g_0^2g_1(y_3 - y_1)\ot
x_0x_1x_2 \\
&\, + \lambda_2g_0g_1^2(y_0  -y_3)\ot x_1x_2x_1 +
\lambda_2g_0^2g_3(y_2- y_3)\ot x_0x_1x_0\\
&\,+ \lambda_2g_1g_3(y_0y_2- y_1y_2-y_0y_1)\ot x_2x_1+ 2\lambda_2\lambda_1g_1g_3\ot x_2x_1
\\
&\,+ \lambda_2g_0g_3(y_2y_1+y_1y_2)\ot x_1x_0  +\lambda_2g_0g_2(y_1y_0+y_0y_1)\ot
x_0x_2\\
&\,+ \lambda_2g_0g_1(y_2y_3-y_2y_1-2y_1y_3 +y_1y_0 +y_0y_3)\ot x_1x_2\\
&\,+\lambda_2g_0g_1(y_0y_2+ 2 y_2y_3-y_2y_1-y_1y_3-y_0y_3)\ot x_0x_1 \\
&\, +\lambda_2g_0(y_2y_1y_0-y_1y_2y_1-y_1y_0y_3-y_0y_1y_3)\ot x_0\\
&\,  + \lambda_2g_1(y_0y_2y_1+y_0y_1y_2-y_2y_1y_3-y_1y_2y_3)\ot x_1\end{align*}
\begin{align*}
&\,  + \lambda_2g_2(y_1y_2y_3 + y_1y_0y_2 - y_0y_2y_3 + y_0y_1y_3-y_0y_1y_0)\ot x_2\\
&\, + \lambda_2(2\lambda_1-\lambda_2)g_0g_3\ot x_1x_0
+\lambda_2(2\lambda_1-\lambda_2)g_0g_2\ot x_0x_2\\
&\, + \lambda_2(\lambda_2-2\lambda_1)g_0(y_3-y_2)\ot
x_0 + \lambda_2g_0(\lambda_2 y_1-\lambda_1y_0)\ot x_0\\
&\,  +
\lambda_2(2\lambda_1-\lambda_2)g_1(y_0 -y_3)\ot
x_1-\lambda_1\lambda_2g_2( y_2+ y_3 )\ot x_2 \\
&+\lambda_2g_2(2\lambda_1-\lambda_2)y_1\ot x_2+\lambda_2g_2( \lambda_2y_0-\lambda_1
y_3)\ot x_2.
\end{align*}
If $z_1= z- \lambda_2 z' + \lambda_2(\lambda_2
-\lambda_1)z''$ where $z'=y_1y_0y_2y_1+ y_0y_2y_1y_0+y_2y_1y_0y_2$, $z''=y_2y_1  + 
y_1y_0 + y_0y_2$, we get $\delta_R^2(z_1t_z^{-1})=z_1t_z^{-1}\ot t_z^{-1}+1\ot z$. Thus,
$\gamma_2(zt_z^{-1})=z_1t_z^{-1}$ and $\gamma_2(z)=z_1$. The computation of
$\delta_L^2(\gamma_2(z))-1\ot \gamma_2(z)$ yields $s_X$ in Subsection \ref{subsec:liftingsf4}.
\end{exa}

\end{document}